\documentclass[12pt]{amsart}
\usepackage{amsfonts, amstext, amsmath, amsthm, amscd, amssymb}
\usepackage[hmargin=2.5cm,vmargin=2.5cm]{geometry}
\usepackage[pdftex]{graphicx, color}
\usepackage{soul}

\usepackage[hidelinks,pagebackref]{hyperref}
\usepackage{enumitem}
\usepackage{cite}
\usepackage{setspace}
\usepackage{float}
\usepackage{subcaption}
\usepackage{tikz, tikz-cd}
\usepackage{mathrsfs}
\usepackage{pinlabel}
\usetikzlibrary{positioning}

\usepackage{xcolor}
\definecolor{darkred}{rgb}{0.4, 0.0, 0.0}

\AtBeginDocument{
   \def\MR#1{}
}

\pagecolor{white}

\allowdisplaybreaks

\newcommand*\unknot{\vcenter{\hbox{\includegraphics[width=2em]{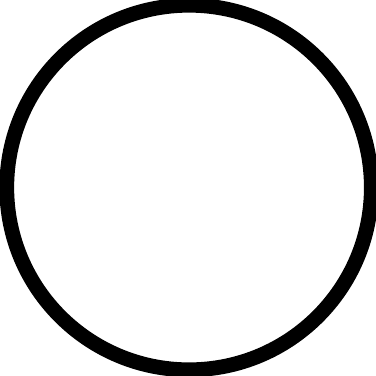}}}}
\newcommand*\astate{\vcenter{\hbox{\includegraphics[width=2em]{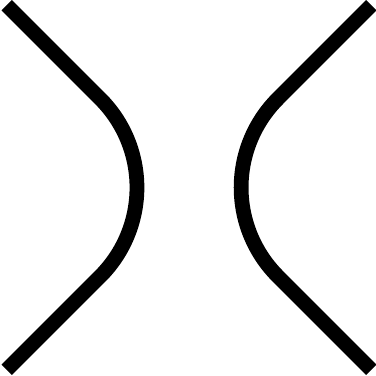}}}}
\newcommand*\bstate{\vcenter{\hbox{\includegraphics[width=2em]{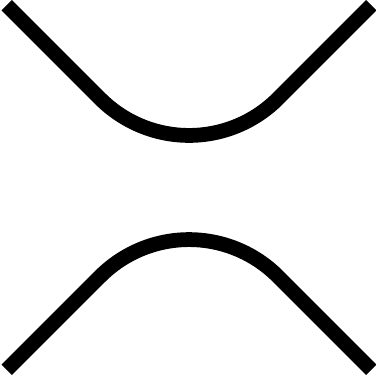}}}}
\newcommand*\pcross{\vcenter{\hbox{\includegraphics[width=2em]{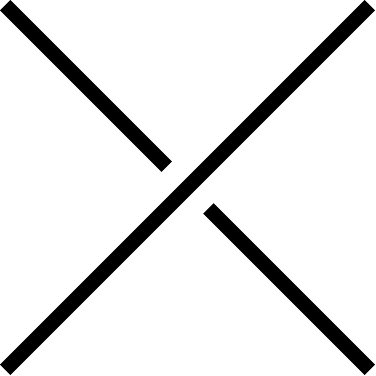}}}}
\newcommand*\dncurl{\vcenter{\hbox{\includegraphics[width=2em]{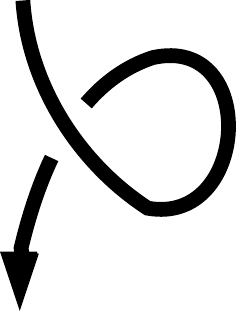}}}}
\newcommand*\dpcurl{\vcenter{\hbox{\includegraphics[width=2em]{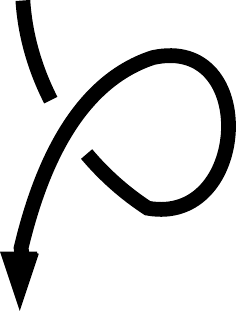}}}}
\newcommand*\dstrand{\vcenter{\hbox{\includegraphics[width=1em]{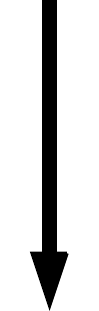}}}}

\newcommand{\Z}{\mathbb{Z}}
\newcommand{\N}{\mathbb{N}}
\newcommand{\R}{\mathbb{R}}

\newcommand{\C}{\mathbb{C}}
\newcommand{\id}{\text{Id}}

\newtheorem{thm}{Theorem}
\numberwithin{thm}{section}

\newtheorem{conj}[thm]{Conjecture}
\newtheorem{prop}[thm]{Proposition}
\newtheorem{lemma}[thm]{Lemma}
\newtheorem{cor}[thm]{Corollary}

\newtheorem*{namedtheorem}{\theoremname}
\newcommand{\theoremname}{testing}
\newenvironment{named_thm}[1]{\renewcommand{\theoremname}{#1}\begin{namedtheorem}}{\end{namedtheorem}}
\newcommand{\refthm}[1]{Theorem~\ref{thm:#1}}
\newenvironment{named_prop}[1]{\renewcommand{\theoremname}{#1}\begin{namedtheorem}}{\end{namedtheorem}}
\newcommand{\refprop}[1]{Proposition~\ref{thm:#1}}

\theoremstyle{definition}
\newtheorem{defn}[thm]{Definition}

\theoremstyle{definition}
\newtheorem*{nameddef}{\defname}
\newcommand{\defname}{testing}
\newenvironment{named_def}[1]{\renewcommand{\defname}{#1}\begin{nameddef}}{\end{nameddef}}
\newcommand{\refdef}[1]{Definition~\ref{def:#1}}

\theoremstyle{definition}

\theoremstyle{remark}
\newtheorem{rmk}[thm]{Remark}

\begin{document}
\title[A Quantum Invariant of Links in $T^2 \times I$ with Volume Conjecture Behavior]{A Quantum Invariant of Links in $T^2 \times I$ \\ with Volume Conjecture Behavior}
\author[Joe Boninger]{Joe Boninger}
\address{Department of Mathematics,  The Graduate Center, City University of New York, New York, NY}
\email{jboninger@gradcenter.cuny.edu}


\maketitle

\begin{abstract}
We define a polynomial invariant $J_n^T$ of links in the thickened torus. We call $J^T_n$ the $n$th toroidal colored Jones polynomial, and show it satisfies many properties of the original colored Jones polynomial. Most significantly, $J_n^T$ exhibits volume conjecture behavior.  We prove the volume conjecture for the 2-by-2 square weave, and provide computational evidence for other links.  We also give two equivalent constructions of $J_n^T$, one as a generalized operator invariant we call a pseudo-operator invariant, and another using the Kauffman bracket skein module of the torus. Finally, we show $J^T_n$ produces invariants of biperiodic and virtual links. To our knowledge, $J^T_n$ gives the first example of volume conjecture behavior in a virtual (non-classical) link.
\end{abstract}

\section{Introduction}

A growing body of evidence supports the idea that the asymptotic growth rate of quantum invariants of links and 3-manifolds encodes geometric information.  This hypothesis was initiated by the well-known Volume Conjecture of Kashaev, Murakami and Murakami.
\begin{conj}[\cite{k97, mm01}]
\label{thm:vol_conj}
For a knot $K \subset S^3$, let $J_n(K;e^{2\pi i/n})$ be the $n$th colored Jones polynomial of $K$ evaluated at $e^{2\pi i/n}$. Then
$$
\lim_{n \to \infty} \frac{2\pi}{n} \log | J_n(K;e^\frac{2\pi i}{n})| = {\rm Vol}(S^3 \setminus K).
$$
\end{conj}
Here Vol indicates \emph{simplicial volume}, which we define to be the sum of the hyperbolic volumes of the hyperbolic pieces in the Jaco-Shalen-Johansson decomposition of $S^3 \setminus K$ (see \cite{s81}). We say a quantum invariant exhibits \emph{volume conjecture behavior} if theoretical or computational evidence supports a conjectured limit as above.

Conjecture \ref{thm:vol_conj} has been generalized to other 3-manifolds in several ways. In \cite{c07}, Costantino extended the colored Jones polynomial to links in $\#_k(S^2 \times S^1)$ using Turaev's theory of shadows and proved the volume conjecture for an infinite family of hyperbolic links. More recently, Chen and Yang \cite{cy18} discovered volume conjecture behavior exhibited by the Witten-Reshetikhin-Turaev and Turaev-Viro invariants of 3-manifolds, two quantum invariants closely related to the colored Jones polynomial. These conjectures have been verified in many cases \cite{dky18, o18}.

In this paper we define a polynomial invariant $J^T_n$, $n \in \N$, of oriented links in the thickened torus, $T^2 \times I$. We call $J^T_n$ the \emph{$n$th toroidal colored Jones polynomial}, and show it satisfies many properties of the colored Jones polynomial for links in $S^3$. For example, we give one construction of $J^T_n$ using the theory of operator invariants, and another using the Kauffman bracket skein module of $T^2 \times I$. Significantly, $J^T_n$ is the first example of volume conjecture behavior in the Kauffman bracket skein module of a manifold other than $S^3$. 

We state the volume conjecture for $J^T_n$ precisely as follows.
\begin{conj}
\label{thm:conj}
For any link $L \subset T^2 \times I$ such that $(T^2 \times I) \setminus L$ is hyperbolic,
$$
\lim_{n \to \infty} \frac{2\pi}{n} \log | J^T_n(L;e^\frac{2\pi i}{n})| = \rm{Vol}((T^2 \times I) \setminus L).
$$
\end{conj}

Here the simplicial volume Vol is simply the hyperbolic volume of $(T^2 \times I) \setminus L$. We prove Conjecture \ref{thm:conj} for the $2$-by-$2$ square weave $W \subset T^2 \times I$ shown in Figure~\ref{fig:square_weave}.
\begin{named_thm}{\refthm{main_one}}
$$
\lim_{n \to \infty} \frac{2\pi}{n} \log | J^T_n(W;e^\frac{2\pi i}{n})| = 4 v_\text{oct} = \rm{Vol}((T^2 \times I) \setminus W)
$$
where $v_\text{oct} \approx 3.6638$ is the volume of the regular ideal hyperbolic octahedron.
\end{named_thm}

\begin{figure}[H]
\centering
\includegraphics[height=4cm]{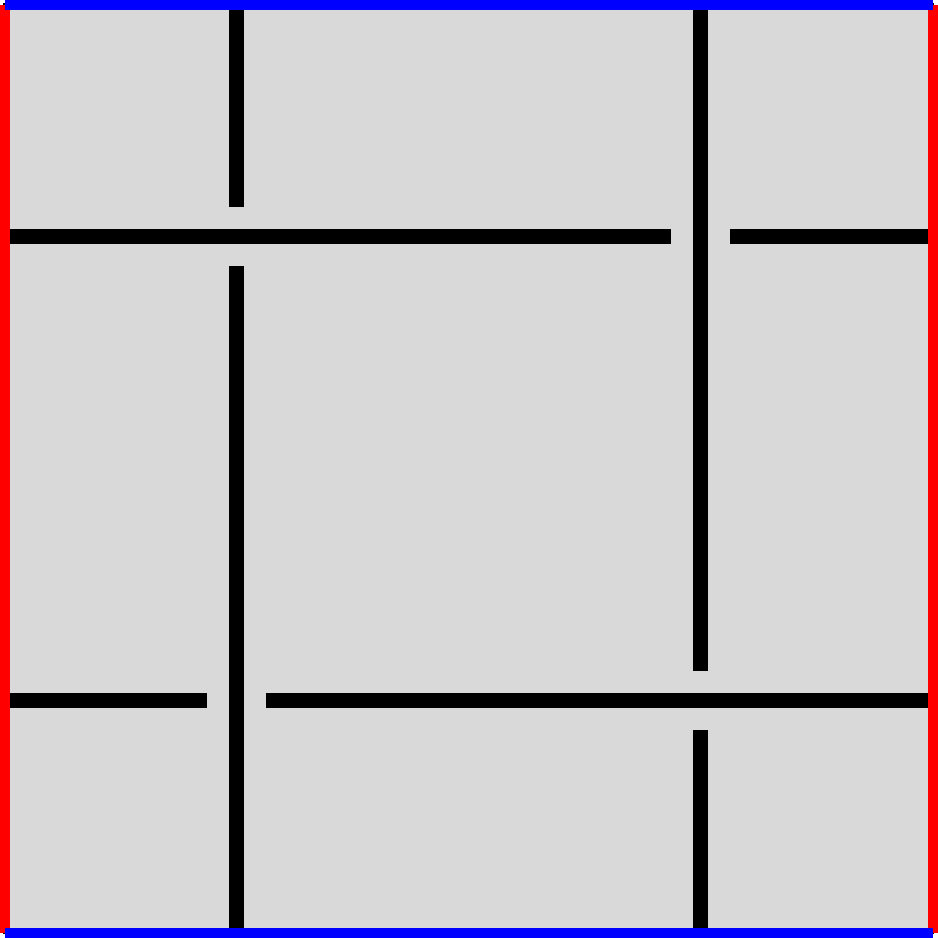}
\caption{The $2$-by-$2$ square weave (opposite sides of the diagram are identified)}
\label{fig:square_weave}
\end{figure}

In addition to Theorem \ref{thm:main_one}, the computations in Table \ref{tab:table} support our volume conjecture. Each row gives the normalized log of the modulus of the toroidal colored Jones polynomial of a certain link, at the relevant root of unity, for different values of $n$. The first four rows are genus one virtual knots in Green's table \cite{g04}---each of these corresponds to a knot in $T^2 \times I$ \cite{k03} with volume computed in \cite{a19}. In the fifth and sixth rows, $B$ and $\ell$ refer respectively to the virtual 2-braid and triaxial weave shown in Figure \ref{fig:links}. (The geometry of $\ell$ is discussed in \cite{ckp19}.)  Finally, $v_\text{tet} \approx 1.0149$ is the volume of the regular ideal hyperbolic tetrahedron.
\begin{table}[H]
\centering\renewcommand\arraystretch{1.2}
\begin{tabular}{|c|cccccc|c|}
\hline
&  \multicolumn{6}{c|}{$(2\pi/n) \cdot \log | J^T_n(L;e^{2\pi i/n})|$ at $n = $} & \\
\hline
Link & 10 & 20 & 30 & 50 & 75 & 100 & Vol \\
\hline
2.1 & 5.4685 & 5.5004 & 5.4843 & 5.4548 & 5.4309 & 5.4215 & 5.3335 \\
3.2 & 7.5047 & 7.6976 & 7.7393 & 7.7566 & 7.7564 & 7.7528 & 7.7069 \\
3.5 & 5.9817 & 6.2649 & 6.3345 & 6.3733 & 6.3836 & 6.3852 & 6.3545 \\
3.7 & 9.0885 & 9.3732 & 9.4523 & 9.5017 & 9.5182 & 9.5231 & 9.5034 \\
$B$ & 7.1834 & 7.3637 & 7.3903 & 7.3953 & 7.3891 & 7.3825 & $2v_\text{oct} \approx 7.3278$ \\
$\ell$ & 9.5569 & 9.9321 & 10.0405 & 10.1130 & 10.1411 & 10.1519 & $10v_\text{tet} \approx 10.149$ \\
\hline
\end{tabular}
\caption{Computational evidence for Conjecture \protect\ref{thm:conj}}
\label{tab:table}
\end{table}

\begin{figure}
\centering
\subcaptionbox{The virtual 2-braid $B$ \label{fig:braid}}{
\includegraphics[height=4cm]{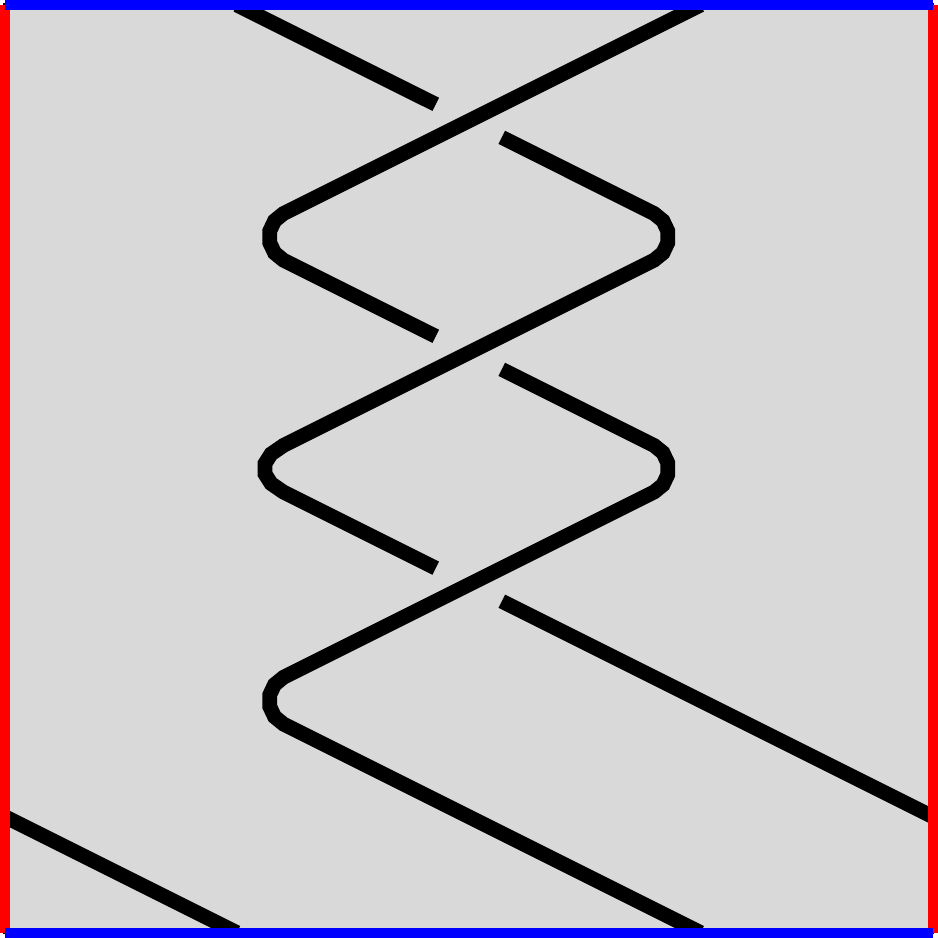}
}
\hspace{1cm}
\subcaptionbox{The triaxial weave $\ell$ \label{fig:weave}}{
\includegraphics[height=4cm]{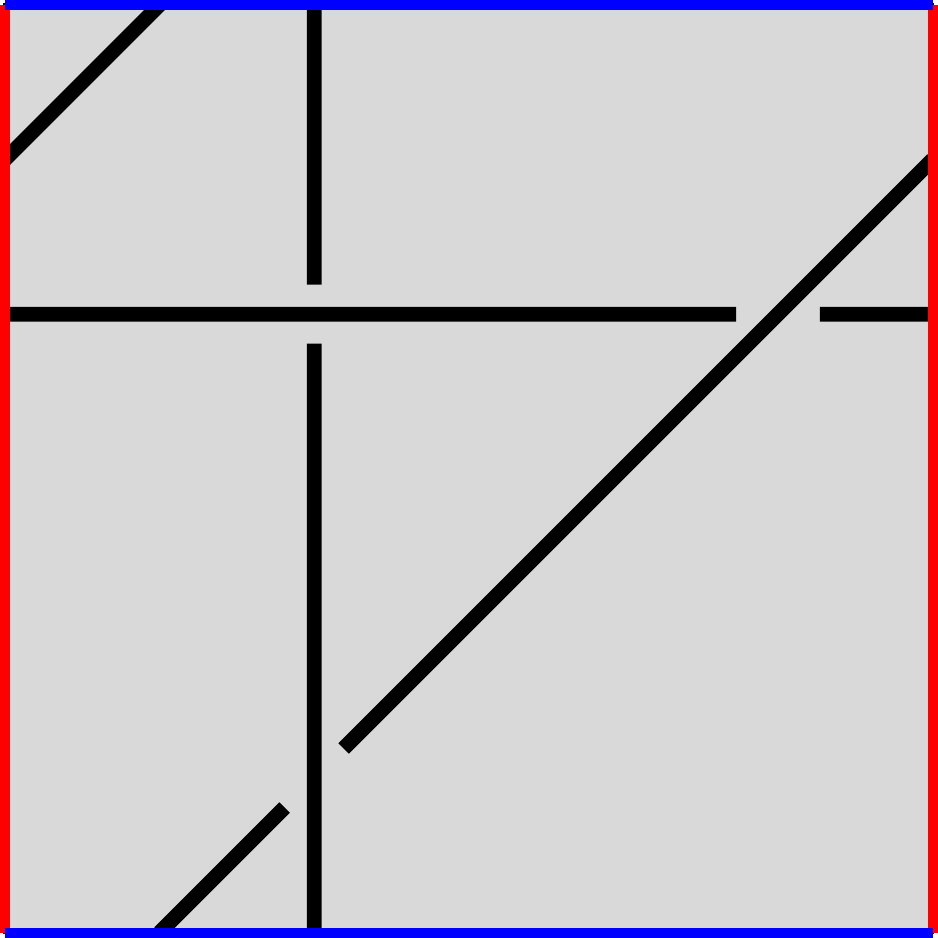}
}
\caption{}
\label{fig:links}
\end{figure}

\subsection*{Pseudo-Operator Invariant}

Volume conjecture behavior is not the only interesting feature of $J^T_n$. Like the original colored Jones polynomial, $J^T_n$ is defined using the theory of operator invariants and the quantum group $\mathscr{A} = U_q(sl(2, \C))$, the quantized universal enveloping algebra of $sl(2, \C)$ specialized to a root of unity $q$. Briefly, given a link $L \subset T^2 \times I$ with diagram $D \subset T^2$, we use the flat geometry of $T^2$ to label certain points of $D$ as critical points. We then assign $\mathscr{A}$-linear operators to each critical point of $D$ and use these local assignments to compute $J_n^T$ as a state sum. This is similar to the construction of the colored Jones polynomial of links in $S^3$, with a key conceptual difference: in $S^3$, the local assignments of $\mathscr{A}$-linear operators to critical points (crossings and local extrema) extend to a global assignment of a single $\mathscr{A}$-linear operator to the entire link. With $J_n^T$ no global assignment is possible, and for this reason we refer to $J^T_n$ as a \emph{pseudo-operator invariant}. The theory of pseudo-operator invariants, which generalizes the theory of operator invariants, may have applications beyond the invariant $J^T_n$. In Section 3, we develop this theory in detail and in the process construct another invariant $\hat{J}^T_{{\bf n}, q}$ of framed, unoriented links in $T^2 \times I$. The invariant $\hat{J}^T_{{\bf n}, q}$ is analogous to the invariant $J_{L, {\bf n}}$ of \cite{km91}, where ${\bf n} = (n_1, \dots, n_k)$ is a multi-integer indicating an integer $n_i$ assigned to each component of $L$.

\subsection*{Skein Module Invariant}

We also consider an $SU(2)$ toroidal colored Jones polynomial obtained by specializing to the quantum group $SU(2)_q$. We show that if $C \subset T^2$ is a contractible, simple closed curve, the level two $SU(2)$ invariant $\hat{J}^T_{{\bf 2}, q} = \hat{J}^T_{(2, \dots, 2),q}$ satisfies
$$
\hat{J}^T_{{\bf 2},q}(C) = -q^{1/2} - q^{-1/2}.
$$
If $C \subset T^2$ is a simple closed curve which is not contractible,
$$
\hat{J}^T_{{\bf 2},q}(C) = 2.
$$
Additionally, we prove $\hat{J}^T_{{\bf 2},q}$ satisfies the Kauffman bracket skein relation. These observations motivate the following definition and theorem, which characterize $\hat{J}^T_{{\bf 2}, q}$ skein-theoretically.

\pagebreak

\begin{named_def}{\refdef{kauffman}}
Define a Kauffman-type bracket $\langle * \rangle_\tau \in \Z[A^{\pm 1}, z]$ on link diagrams in $T^2$ (and framed links in $T^2 \times I$) by the relations
\begin{enumerate}[label=(\alph*)]
\item $\langle \varnothing \rangle_\tau= 1$.
\item Let $C \subset T^2$ be a simple closed curve disjoint from a diagram $D \subset T^2$. 
\begin{enumerate}[label=(\roman*)]
\item If $C$ is contractible, $\langle C \sqcup D \rangle_\tau = (-A^2 - A^{-2}) \langle D \rangle_\tau$.
\item If $C$ is not contractible, $\langle C \sqcup D \rangle_\tau = z \langle D \rangle_\tau$.
\end{enumerate}
\item $\langle \pcross \rangle_\tau = A \langle \astate \rangle_\tau + A^{-1} \langle \bstate \rangle_\tau$.
\end{enumerate}
\end{named_def}

Here $A$ and $z$ are indeterminates.

\begin{named_thm}{\refthm{kauffman}}
For any framed link $L \subset T^2 \times I$,
$$
\hat{J}^T_{{\bf 2},q}(L) = \langle L \rangle_\tau |_{A^4 = q, z = 2}.
$$
As a corollary, for any oriented, unframed link $L \subset T^2 \times I$ with diagram $D \subset T^2$,
$$
J^T_2(L;q) =  [-A^{-3 w(D)} \langle D \rangle_\tau]|_{A^4 = q, z = 2}.
$$
\end{named_thm}
This gives a skein-theoretic construction of the toroidal Jones polynomial generalizing that of the usual Jones polynomial. In fact, our Theorem \ref{thm:skein_op} and Corollary \ref{thm:skein_cjp} prove much stronger statements defining $\hat{J}^T_{{\bf n}, q}$ and $J^T_n$ skein-theoretically for all ${\bf n}$ and $n$; to accomplish this we use the Kauffman bracket skein module of the thickened torus.

\subsection*{Why $z = 2$?}

Relations (a), (b), and (c) in Definition \ref{def:kauffman} are identical to the relations defining the standard Kauffman bracket \cite{k87}, with the additional stipulation in (b)(ii) that essential, simple closed curves can be removed from a diagram by multiplying by $z$. (A somewhat similar bracket is defined in \cite{k11}.) To obtain Theorem \ref{thm:kauffman}, and for a geometrically motivated theory, it is necessary to fix $z = 2$. Indeed, only when $z = 2$ do we obtain an $R$-matrix, allowing us to do calculations as in Table \ref{tab:table}. Proposition \ref{thm:why_two} below shows \emph{any} pseudo-operator invariant takes the value $2$ on essential, simple closed curves in $T^2$, if those curves have been colored by a $2$-dimensional representation of a quantum group. In Appendix A we examine this property further using rotation number and Lin and Wang's definition of the usual Jones polynomial \cite{lw01}---see Proposition \ref{thm:torus_knots} and the following discussion.

\subsection*{Comparison with $J_n$}

For any link $L$  in $T^2 \times I$, there exists a link $\hat{L} \subset S^3$ such that $(T^2 \times I) \setminus L$ and $S^3 \setminus \hat{L}$ are homeomorphic: $\hat{L}$ has a Hopf sublink $H$ whose components are the cores of the tori which make up $S^3 \setminus (T^2 \times I)$ (see Figure \ref{fig:ts_links}). We show $J^T_n(L)$ and $J_n(\hat{L})$ are fundamentally distinct invariants.

A key difference between the two is that $J^T_n$ is unchanged by orientation-preserving homeomorphisms of the torus:
\begin{named_prop}{\refprop{follow_up}}
If the link diagram $D' \subset T^2$ is obtained from a diagram $D \subset T^2$ by an orientation-preserving homeomorphism of $T^2$, then for the corresponding links $L', L \subset T^2 \times I$, $J^T_n(L';q) = J^T_n(L;q)$ for all $n$.
\end{named_prop}
Using this proposition, we can construct infinite families of non-isotopic links in $T^2 \times I$ with identical toroidal colored Jones polynomials, whose corresponding links in $S^3$ all have distinct colored Jones polynomials. See Figure \ref{fig:ts_links} for a simple example, where the links on the left in $T^2 \times I$ have the same toroidal colored Jones polynomials, but the corresponding links on the right in $S^3$ have different colored Jones polynomials.

While this makes $J^T_n$ a less sensitive invariant than $J_n$, it also makes $J^T_n$ applicable in a wider range of contexts. In Section 8, for example, we show $J^T_n$ gives invariants of virtual links and biperiodic links. To our knowledge, $J^T_n$ is the first invariant of virtual links to exhibit volume conjecture behavior in a non-classical setting.

\begin{figure}
\centering
\includegraphics[height=4.5cm]{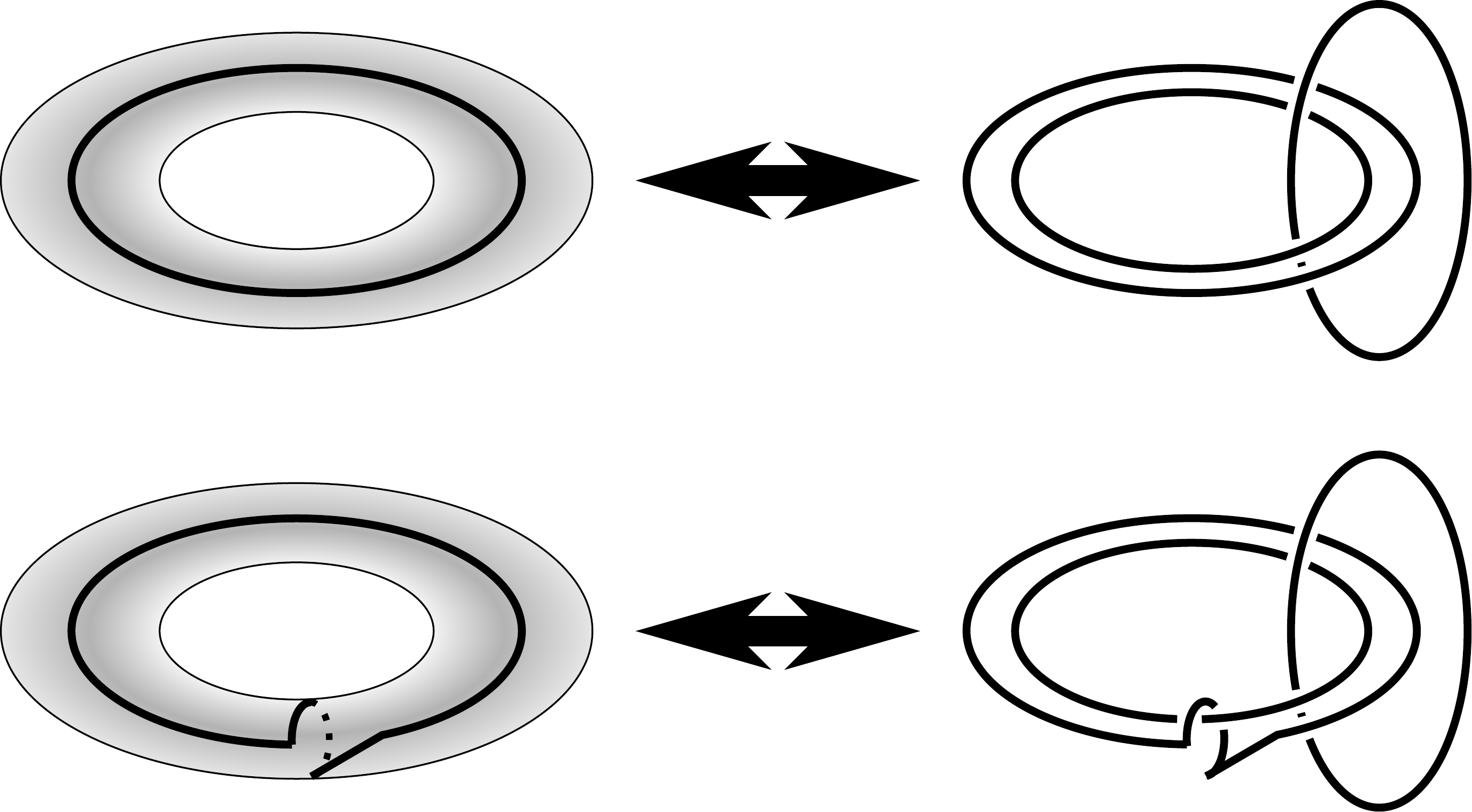}
\caption{Non-isotopic links with homeomorphic complements in $T^2 \times I$ and $S^3$}
\label{fig:ts_links}
\end{figure}

Finally, while $J^T_n(L;q)$ and $J_n(\hat{L};q)$ are different invariants, there is an important special case when the toroidal colored Jones polynomial and usual colored Jones polynomial completely determine each other (see Figure \ref{fig:t_to_t}):

\begin{named_thm}{\refthm{relationship}}
Let $L'$ be a link in $S^3$, and consider an inclusion of $L'$ in an embedded $2$-sphere in $T^2 \times I$. Let $K \subset T^2 \times I$ be a knot projecting to an essential, simple closed curve in $T^2 \times \{0\}$, and let $L$ be a connect sum $L = L' \# K$. Then
$$
J^T_n(L;q) = n \cdot J_n(L';q)
$$
for all $n$.
\end{named_thm}
An immediate corollary of Theorem \ref{thm:relationship} is that, for $L$ and $L'$ as in the theorem,
$$
\lim_{n \to \infty} \frac{2\pi}{n} \log | J^T_n(L;e^\frac{2\pi i}{n})| = \lim_{n \to \infty} \frac{2\pi}{n} \log | J_n(L';e^\frac{2\pi i}{n})|.
$$
In Section 7, we use this fact to prove that a suitable generalization of our Volume Conjecture \ref{thm:conj} implies the original Volume Conjecture \ref{thm:vol_conj}---see Conjecture \ref{thm:broad_conj} and Corollary \ref{thm:cor_imp} below. It is not clear whether the reverse implication is true.

\subsection*{Outline}

The outline of this paper is as follows: in Section 2, we review Kauffman bracket skein modules and operator invariants. In Section 3 we define a general pseudo-operator invariant $\Phi$ of framed links in $T^2 \times I$, and in Section 4 we specialize $\Phi$ to $U_q(sl(2,\C))$ to obtain $J^T_n$ and $\hat{J}^T_{{\bf n}, q}$. In Section 5 we define these invariants skein-theoretically. In Section 6 we prove Theorem \ref{thm:main_one}, and in Section 7 we discuss generalizations of our Volume Conjecture \ref{thm:conj}. In particular, we consider the case of non-hyperbolic links in $T^2 \times I$ and show that a generalization of Conjecture \ref{thm:conj} implies the original Volume Conjecture. In Section 8 we discuss $J^T_n$ as an invariant of biperiodic and virtual links. Finally, in Appendix A, we study the behavior of $J^T_2$ through the lens of Lin and Wang's formulation of the Jones polynomial \cite{lw01}.

\subsection*{Acknowledgements}
We thank Ilya Kofman for his help and guidance with this project, and Hitoshi Murakami, Adam Sikora, and Abhijit Champanerkar for helpful comments.

\section{Background}

\subsection{Kauffman Bracket Skein Modules}
For a $3$-manifold $M$ and indeterminate $A$, let $\mathscr{L}(M)$ be the free $\Z[A^{\pm 1}]$-module generated by regular isotopy classes of framed links in $M$. The \emph{Kauffman bracket skein module} of $M$ \cite{t88, p91}, $\mathscr{S}(M)$, is the quotient of $\mathscr{L}(M)$ by the submodule generated by the following two relations:
\begin{enumerate}[label=(\roman*)]
\item $\unknot \sqcup L = (-A^2 - A^{-2}) L$.
\item $\pcross = A \astate + A^{-1} \bstate$.
\end{enumerate}
The links in each expression above are identical except in a ball where they look as shown, and all diagrams are assumed to have blackboard framing. Each link $L \subset M$ is represented in $\mathscr{S}(M)$ by $\langle L \rangle$, called the \emph{Kauffman bracket} of $L$. If $M = \Sigma \times I$, $\Sigma$ an orientable surface, we also denote the skein module of $M$ by $\mathscr{S}(\Sigma)$. In this case, gluing two copies of $\Sigma \times I$ together along a boundary component gives $\mathscr{S}(\Sigma)$ the structure of a $\Z[A^{\pm 1}]$-algebra.

As an algebra, the skein module $\mathscr{S}(\mathcal{A})$ of the thickened annulus $\mathcal{A} \times I$ is generated by a copy of its core with framing parallel to $\mathcal{A} \times \{0\}$. Sending this core to $z$ gives an algebra isomorphism $\mathscr{S}(\mathcal{A}) \cong \Z[A^{\pm 1}][z]$, so that the set $\{1, z, z^2, z^3, \dots\}$ is a basis of $\mathscr{S}(\mathcal{A})$ as a $\Z[A^{\pm 1}]$-module. An alternate basis for $\mathscr{S}(\mathcal{A})$ is given by the Chebyshev polynomials $S_j(z)$, $j \geq 0$, defined recursively by
\begin{align}
\label{eq:chebyshev}
S_0(z) = 1, & & S_1(z) = z, & & S_{j + 1}(z) = zS_j - S_{j - 1}.
\end{align}

If $L$ is a link in $M$ with $k$ components, we can construct a multilinear map
\begin{equation}
\label{eq:multi-bracket}
\langle \cdots \cdots \rangle_L : \mathscr{S}(\mathcal{A})^{\otimes k} \to \mathscr{S}(M)
\end{equation}
called the \emph{Kauffman multi-bracket}, as follows. For $z^{i_j} \in \Z[A^{\pm 1}][z] \cong \mathscr{S}(\mathcal{A})$, $i_j \geq 0$, let $L^{i_1, \dots, i_k}$ be the framed link in $M$ obtained by cabling the $j$th component of $L$ by $i_j$ parallel copies of itself. Define
$$
\langle z^{i_1}, \dots, z^{i_k}\rangle_L = \langle L^{i_1, \dots, i_k} \rangle
$$
and extend $\Z[A^{\pm 1}]$-multilinearly to all of $\mathscr{S}(\mathcal{A})$. 

Sending the empty link to $1$ gives an isomorphism from $\mathscr{S}(S^3)$ to $\Z[A^{\pm 1}]$. Thus, for a link $L \in S^3$ with $k$ components, the Kauffman multi-bracket is a map
$$
\langle \cdots \cdots \rangle_L : \mathscr{S}(\mathcal{A})^{\otimes k} \to  \Z[A^{\pm 1}].
$$
Let $L$ be an oriented, unframed link in $S^3$ with $k$ components and $D$ a diagram for $L$ with writhe $w(D)$. The \emph{$n$th colored Jones polynomial} of $L$, $J_n(L;q)$, is defined by
\begin{equation}
\label{eq:skein_jones}
J_n(L;q) = \Big[ \frac{((-1)^{n - 1} A^{n^2 - 1})^{-w(D)}}{-A^2 - A^{-2}} \langle S_{n - 1}(z), \dots, S_{n - 1}(z) \rangle_D \Big]\Big|_{A^4 = q}.
\end{equation}

In Section 5 we study the Kauffman bracket skein module $\mathscr{S}(T^2)$ of the thickened torus $T^2 \times I$ and its associated Kauffman multi-bracket. $\mathscr{S}(T^2)$ is generated as an algebra by isotopy classes of simple closed curves in $T^2$, which are in bijection with the set of tuples $(a,b) \in \Z^2$ such that either $a = b = 0$, or $a$ and $b$ are coprime, modulo the relation $(a, b) \sim (-a, -b)$. We think of $(a,b)$ as the curve homotopic to $a$ times a meridian plus $b$ times a longitude, and write $(a,b)^m$ to indicate $m$ parallel copies of such a curve. Additionally, to avoid ambiguity, we denote the image of a link $L \subset T^2 \times I$ in $\mathscr{S}(T^2)$ by $\langle L \rangle_T$ and use $\langle \dots \rangle_{T, L}$ to mean the multi-bracket map determined by $L$,
\begin{equation}
\label{eq:torus_multibracket}
\langle \dots \dots \rangle_{T, L} : \mathscr{S}(\mathcal{A})^{\otimes k} \to \mathscr{S}(T^2).
\end{equation}

\subsection{Tangle Operators}

An alternate definition of the colored Jones polynomial comes from the theory of \emph{tangle operators}. The exposition here follows \cite[Sec.~3]{km91}.

Recall that a \emph{tangle} $T$ is a $1$-manifold properly embedded (up to isotopy) in the unit cube $I^3 \subset \R^3$ with $\partial T \subset \{\frac{1}{2}\} \times I \times \partial I$, and define $\partial_- T = T \cap (I^2 \times \{0\})$ and $\partial_+ T = T \cap (I^2 \times \{1\})$. Choosing a regular projection onto $\{0\} \times I^2$ gives a \emph{tangle diagram} of $T$.

For two tangles $S$ and $T$, denote by $S \otimes T$ the tangle formed by placing $S$ and $T$ side by side so the boundary $I \times \{1\} \times I$ of $S$ equals the boundary $I \times \{0\} \times I$ of $T$. Similarly, by $S \circ T$ we mean the tangle formed by stacking $S$ and $T$ vertically so $\partial_+T = \partial_-S$; this operation can be performed only if $|\partial_+T| = |\partial_-S|$. With these operations, the set of all tangle diagrams is generated by the five {\it elementary diagrams} $I$, $R$, $L$, $\cap$ (called a \emph{cap}), and $\cup$ (called a \emph{cup}) shown in Figure \ref{fig:el_di}. Below, we assume tangles are equipped with orientations and framings.

\begin{figure}[H]
\labellist
\small\hair 2pt
\pinlabel $I$ at 45 20 
\pinlabel $R$ at 165 20
\pinlabel $L$ at 285 20
\pinlabel $\cap$ at 407 20
\pinlabel $\cup$ at 527 20
\endlabellist
\centering
\includegraphics[height=1.75cm]{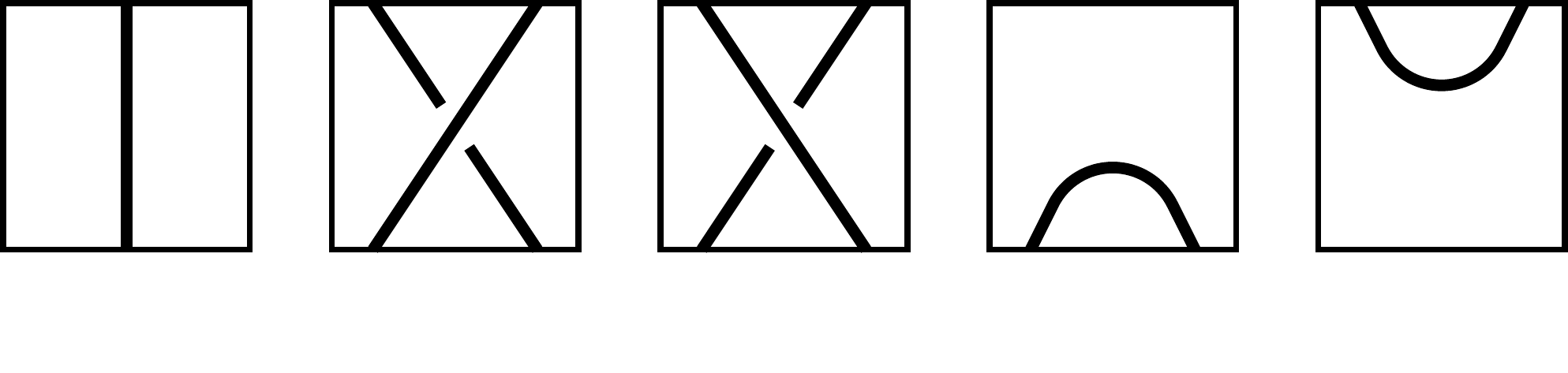}
\caption{Elementary diagrams}
\label{fig:el_di}
\end{figure}
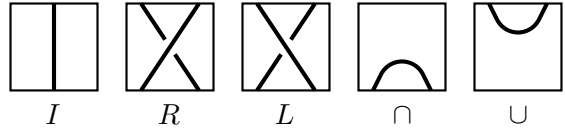

Fix a quasitriangular Hopf algebra $(\mathscr{A}, \breve{R})$ with $R$-matrix $\breve{R} = \sum \alpha_i \otimes \beta_i \in \mathscr{A} \otimes \mathscr{A}$ and define a \emph{$V$-coloring} of a tangle $T$ (or one of its diagrams) to be an assignment of an $\mathscr{A}$-module to each component of $T$. This induces a coloring of $\partial T$ as follows: If $C$ is a component of color $V$, we assign $V$ to each endpoint of $C$ where $C$ is oriented downward and the dual module $V^*$ to each endpoint where $C$ is oriented upward. Tensoring from left to right gives {\it boundary $\mathscr{A}$-modules} $T_\pm$ assigned to $\partial_\pm T$ with the empty tensor product defined to be $\C$.

Suppose $\mathscr{A}$ contains a unit $\mu$ with the following properties:
\begin{enumerate}[label=(\roman*)]
\item $\mu \alpha \mu^{-1} = S^2(\alpha)$ for all $\alpha \in \mathscr{A}$, where $S$ is the antipode of $\mathscr{A}$.
\item $\sum \alpha_i \mu^{-1} \beta_i = \sum \beta_i \mu \alpha_i$.
\end{enumerate}
We call such a unit a \emph{good unit} of $\mathscr{A}$. In this case, by the following fundamental result, any tangle $T$ gives an $\mathscr{A}$-linear map $T_- \to T_+$.

\begin{thm}[\protect\cite{km91, rt90}]
\label{thm:fund_tangles}
There exist unique $\mathscr{A}$-linear operators $\mathcal{F}_T = \mathcal{F}_T^{\mathscr{A}, \breve{R}, \mu} : T_- \to T_+$ assigned to each colored framed tangle $T$ which satisfy $\mathcal{F}_{T \circ T'} = \mathcal{F}_T \circ \mathcal{F}_{T'}$, $\mathcal{F}_{T \otimes T'} = \mathcal{F}_T \otimes \mathcal{F}_{T'}$, and for the tangles given by the elementary diagrams with blackboard framing,
\begin{align*}
\mathcal{F}_I &= \id \\
\mathcal{F}_R &= R \text{ and } \mathcal{F}_L = R^{-1} \\
\mathcal{F}_{\protect\rotatebox[origin=c]{0}{$\curvearrowright$}} &= E \text{ and } \mathcal{F}_{\protect\rotatebox[origin=c]{0}{$\curvearrowleft$}} = E_\mu \\
\mathcal{F}_{\protect\rotatebox[origin=c]{180}{$\curvearrowleft$}} &= N \text{ and } \mathcal{F}_{\protect\rotatebox[origin=c]{180}{$\curvearrowright$}} = N_{\mu^{-1}} 
\end{align*}
where $R = \tau \circ \breve{R}$, $\tau$ the transposition map $\alpha \otimes \beta \mapsto \beta \otimes \alpha$. Additionally $E(f \otimes x) = f(x)$, $E_\mu(x \otimes f) = f(\mu x)$, $N(1) = \sum e_i \otimes e^i$ and $N_{\mu^{-1}}(1) = \sum e^i \otimes (\mu^{-1} e_i)$ for any basis $e_i$.
\end{thm}

The map $R$ is also called an $R$-matrix, and the map $\mathcal{F}_T$ is called the \emph{operator invariant} of $T$. 

\begin{rmk}
The quasitriangular Hopf algebra in this construction can be replaced  more generally with a \textit{ribbon category} \cite{t94}.
\end{rmk}

We set $\mathscr{A} = \mathscr{A}_q = \mathcal{U}_q(sl(2, \C))$, the quantized universal enveloping algebra of $sl(2,\C)$ specialized to $q = e^{2\pi i/r}$ (see \cite{k95}), and fix a certain good unit $\mu$. We also limit tangle colorings to a distinguished set of $\mathscr{A}$-modules $\{V^1, \dots, V^r\}$, $V^n$ coming from the unique $n$-dimensional irreducible representation of $sl(2, \C)$ \cite{rt90}. If $L \subset I^2$ is a $k$-component link and ${\bf n} = (n_1, \dots, n_k)$ a multi-integer, let $L({\bf n})$ be the $\mathscr{A}$-colored link $L$ with $j$th component colored $V^{n_j}$. With this setup, the colored Jones polynomial is defined to be
\begin{equation}
\label{eq:tangle_jones}
J_n(L;q) = (q^{(n^2 - 1)/4})^{-w(D)} \frac{\{1\}}{\{n\}} \mathcal{F}^{\mathscr{A}_q}_{L(n, \dots, n)} = (q^{(n^2 - 1)/4})^{-w(D)} \frac{1}{[n]} \mathcal{F}^{\mathscr{A}_q}_{L(n, \dots, n)}
\end{equation}
where $q = e^{2\pi i/r}$, $r \in \N$, and $w(D)$ is the writhe of the tangle diagram of $L$. The terms $\{m\}$ and $[m]$ are the quantum integers defined by
\begin{align*}
\{m\} &= \{m\}_q = q^{m/2} - q^{-m/2} \\
[m] &= [m]_q = \frac{\{m\}}{\{1\}}.
\end{align*}
The boundary $\mathscr{A}$-modules of any colored link $L({\bf n})$ are both $\C$, so $\mathcal{F}^{\mathscr{A}_q}_{L({\bf n})}$ is a linear map from $\C$ to $\C$---a scalar. This scalar is a Laurent polynomial in $q$.

In fact, the invariant $J_n$ as defined in (\ref{eq:tangle_jones}) is not strictly equal to $J_n$ as defined in (\ref{eq:skein_jones})---for example, the two definitions differ by a sign on the two-component unlink. Achieving precise equality requires a normalization of (\ref{eq:tangle_jones}) equivalent to specializing $\mathcal{F}$ to the quantum group $SU(2)_q$ rather than $U_q(sl(2, \C))$ \cite{km91, m93}. For this reason, we refer to the invariant (\ref{eq:skein_jones}) as the \emph{$SU(2)$ colored Jones polynomial}.

In the following section we generalize the theory of tangle operators to links in $T^2 \times I$, leading in Section 4 to the definition of the toroidal colored Jones polynomial.

\section{Pseudo-Operator Invariants}

For a link $L$ in the thickened torus $T^2 \times I$, we take a regular projection to $T^2 \times \{0\}$ to obtain a link diagram $D \subset T^2$. Let $\pi : \R^2 \to T^2$ be a smooth, orientation-preserving covering map with fundamental domain the unit square $I^2 \subset \R^2$ and deck transformations generated by horizontal and vertical unit shifts of $\R^2$. Let $\tilde{D} = \pi^{-1}(D)$. Define $p \in D$ to be a \emph{local extremum} of $D$, so that a small neighborhood of $p$ is a cap or cup, if $p$ has a lift $\tilde{p}$ which is a local extremum of $\tilde{D}$ with respect to the height function $(x,y) \mapsto y$ on $\R^2$.

\begin{defn}
A point $p \in D$ is a \emph{critical point} of $D$ if it is a local extremum or crossing point. A \emph{torus diagram} $D \subset T^2$ is a regular projection of a smooth link $L \subset T^2 \times I$ onto $T^2 \times \{0\}$, such that critical points are isolated.
\end{defn}

Below, all diagrams in $T^2$ are assumed to be torus diagrams.

Fix a quasitriangular Hopf algebra $(\mathscr{A}, R)$ and good unit $\mu \in \mathscr{A}$. As in Section 2.2, a $V$-\emph{coloring} (or simply \emph{coloring}) of $D$ is an assignment of an $\mathscr{A}$-module to each link component.

We now define an invariant $\Phi_\pi = \Phi^{\mathscr{A}, R, \mu}_\pi$ of oriented $V$-colored link diagrams in $T^2$ with framing parallel to $T^2 \times \{0\}$. Let $D \subset T^2$ be such a diagram and $P$ the set of critical points of $D$. For each $p \in P$, there exists a small rectangular neighborhood $T(p)$ of $p$ and a local section $\psi$ of $\pi$, $\psi : T(p) \to \R^2$, giving $T(p)$ the structure of an oriented, blackboard-framed, elementary tangle diagram. In this way Theorem \ref{thm:fund_tangles} assigns an $\mathscr{A}$-linear operator $\mathcal{F}_{T(p)}$ to each $T(p)$, $p \in P$, with boundary $\mathscr{A}$-modules $T(p)_\pm$. We cannot generally extend these local assignments to a global assignment of an $\mathscr{A}$-linear operator to $D$, as Theorem \ref{thm:fund_tangles} does for tangle diagrams in $I^2$. However, the local assignments of operators to each critical point still allow us to give $D$ a value in $\C$ using the {\it state sum} formulation of the theory, as explained below.

For each $\mathscr{A}$-module $V$, fix a basis $B_V$ of $V$ as a $\C$-vector space. Removing the set of critical points $P$ from $D$ breaks it into components, each colored by some $V$ and each oriented upward or downward when lifted to $\R^2$. A {\it state} $\sigma$ is an assignment of a label $\sigma(S)$ to each component $S$ of $D \setminus P$ as follows: if $S$ is colored by the module $V$ and oriented downward, $\sigma(S)$ is an element of $B_V$. If $S$ is oriented upward, $\sigma(S)$ is an element of the dual basis $B_{V^*}$.

A state $\sigma$ determines a \emph{weight} $\omega_p(\sigma)$ of each critical point. For each $p \in P$, taking tensor products of the labels $\sigma(S)$ of the strands above and below $p$ gives basis elements $\sigma(p)_\pm$ of the modules $T(p)_\pm$. Define the weight $\omega_p(\sigma) \in \C$ to be the coefficient of $\sigma(p)_+$ in $\mathcal{F}_{T(p)}(\sigma(p)_-)$, and define the weight of the state $\sigma$ by
\begin{equation}
\label{eq:state_sum_1}
\omega(\sigma) = \prod_{p \in P} \omega_p(\sigma)
\end{equation}
where the empty product (if $D$ contains no critical points) is defined to be $1$. Finally, set
\begin{equation}
\label{eq:state_sum_2}
\Phi_\pi(D) = \sum_{\sigma} \omega(\sigma),
\end{equation}
where the sum is over all states of $D$.

For an example computation of the weight of a critical point, let $p$ be the crossing point with neighborhood $T(p)$ shown in Figure \ref{fig:crit_point}. Viewing $T(p)$ as a tangle diagram, both tangle components of $T(p)$ are colored by the same $\mathscr{A}$-module $V$, so $T(p)_- = T(p)_+ = V \otimes V$. Additionally, since $T(p)$ is a positive crossing, $\mathcal{F}_{T(p)} = R$, viewed as a map from $V \otimes V$ to itself. In the given state $\sigma$, the diagram components of $T(p) \setminus p$ are assigned basis elements $e_i, e_j, e_k, e_l \in V$ as shown, where $\{e_0, e_1, \dots, e_{n - 1}\}$ is a basis for $V$. We have $\sigma(p)_- = e_k \otimes e_l$, $\sigma(p)_+ = e_i \otimes e_j$, and if $R$ satisfies
$$
R(e_k \otimes e_l) = R^{0,0}_{kl} (e_0 \otimes e_0) + \dots + R^{ij}_{kl} (e_i \otimes e_j) + \dots + R^{n - 1, n - 1}_{kl} (e_{n - 1} \otimes e_{n - 1}),
$$
where the $R^{st}_{kl} \in \C$ are scalars, $0 \leq s,t \leq n - 1$, then $\omega_p(\sigma) = R^{ij}_{kl}$.
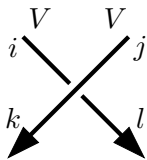
\begin{figure}[H]
\labellist
\small\hair 2pt
\pinlabel $V$ at 35 140 
\pinlabel $V$ at 110 140
\pinlabel $i$ at 8 110
\pinlabel $j$ at 135 110
\pinlabel $k$ at 8 40
\pinlabel $l$ at 135 40
\endlabellist
\centering
\includegraphics[height=2cm]{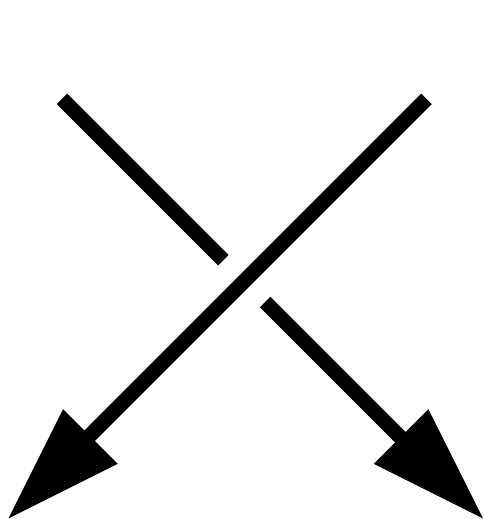}
\caption{A $V$-colored tangle $T(p)$ near a crossing point $p$, labelled by a state $\sigma$}
\label{fig:crit_point}
\end{figure}

\begin{lemma}
\label{thm:basis_invar}
The value $\Phi_\pi(D)$ does not depend on the choice of bases of the colors $V$. 
\end{lemma}

\begin{proof}
To prove the lemma we give an alternate construction of $\Phi_\pi$. Recall $I^2$ is a fundamental domain for the covering map $\pi : \R^2 \to T^2$; adjusting $D$ if necessary, we assume no critical points of $\tilde{D} = \pi^{-1}(D)$ occur in $\partial I^2$. Shift $\tilde{D} \cap (\{1\} \times I)$, the points of $\tilde{D}$ intersecting the right side of $I^2$, slightly upward by an isotopy of $I^2$ which is the identity on $\{0\} \times I$ and does not change the set of critical points of $\tilde{D}$. Because the left and right sides of $I^2$ are identified in the torus, each point $q \in \tilde{D} \cap (\{0\} \times I)$ has a corresponding point $q' \in \tilde{D} \cap (\{1\} \times I)$, slightly higher than $q$ as a result of the isotopy. As a final step of the construction, connect each $q$ and $q'$ by a curve $c_q : I \to I^2$ which satisfies $c(0) = q$, $c(1) = q'$, and is monotonically increasing in height. This produces a {\it virtual tangle diagram} $D' = (I^2 \cap \tilde{D}) \cup_{q \in D \cap (\{0\} \times I)}c_q(I)$ whose classical (i.e.~non-virtual) critical points are the same as the critical points of $D$ and whose virtual crossings are any point where some $c_q(I)$, $q \in \tilde{D} \cap (\{0\} \times I)$, intersects another point of $D'$. See Figure \ref{fig:basis_construct} for an example, where the virtual crossing on the right is circled. We assume virtual crossings are isolated from other critical points.

\begin{figure}[H]
\centering
\includegraphics[height=4cm]{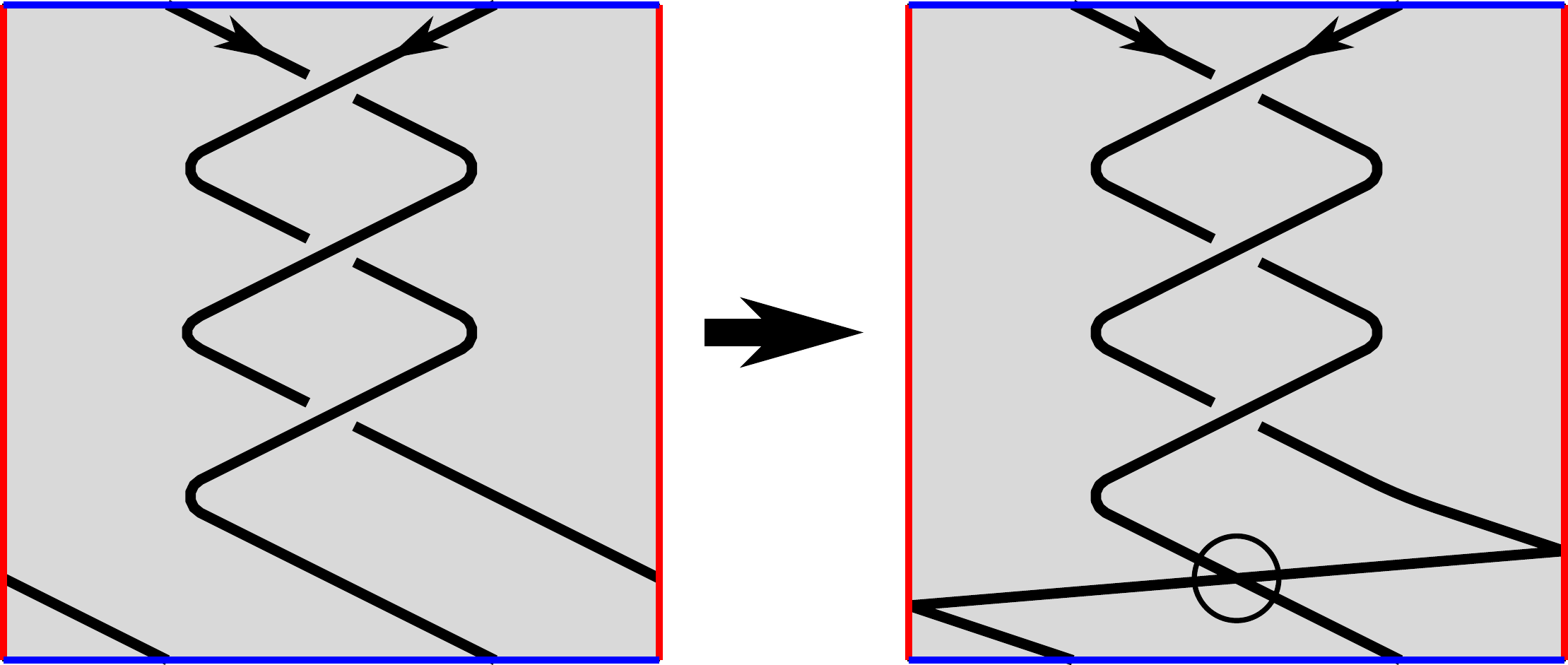}
\caption{Constructing a virtual tangle diagram from a torus diagram}
\label{fig:basis_construct}
\end{figure}

The coloring of $D$ induces a coloring of $D'$ in an obvious way.  As before, let $P$ be the set of (classical and virtual) critical points of $D'$ with $T(p)$ a small rectangular neighborhood of $p \in P$. If $p \in P$ is a classical critical point, the functor $\mathcal{F}$ of Theorem \ref{thm:fund_tangles} associates an $\mathscr{A}$-linear operator $\mathcal{F}_{T(p)}$ to $T(p)$ which agrees with the operator assigned to $T(p)$ in the construction of $\Phi_\pi$. If $p \in P$ is a virtual crossing, define $\mathcal{F}_{T(p)}$ to be the transposition map $\tau : \alpha \otimes \beta \mapsto \beta \otimes \alpha$ as in \cite{k99}. (This is a $\C$-linear map but not generally an $\mathscr{A}$-linear one.) Because $\partial_-D'$ and $\partial_+D'$ are identified in the torus, $D'_- = D'_+ = V$ for some $\mathscr{A}$-module $V$. Thus, extending the local operator assignments $\mathcal{F}_{T(p)}$, $p \in P$, as in Theorem \ref{thm:fund_tangles} associates $D'$ with a $\C$-linear map $\phi : V \to V$. Define $\Phi_\pi'(D') = \text{Tr}(\phi)$; we claim $\Phi_\pi'(D') = \Phi_\pi(D)$.

Computing $\Phi_\pi'(D')$ as a state sum, as in \cite{km91, k99, my18}, shows the two invariants agree. Fix a basis for each $\mathscr{A}$-module. As in the construction of $\Phi_\pi$, a state is an assignment of basis elements to components of $D' \setminus P$ and the weight of a state is the product of the weights of the critical points. Taking the trace of $\phi$ ensures identified strands of $D'_-$ and $D'_+$ are assigned the same basis element in any state with nonzero weight. If the strands near a virtual crossing $p$ are assigned basis elements $e_i, e_j, e_k, e_l$, as in Figure \ref{fig:virtual_state}, the weight of $p$ is $\delta_i^l\delta_j^k$, $\delta$ the Kronecker delta. This ensures the identified strands on either side of $p$ have the same state, in which case $p$ has weight $1$. 

\begin{figure}[H]
\labellist
\small\hair 2pt
\pinlabel $i$ at 8 110
\pinlabel $j$ at 135 110
\pinlabel $k$ at 8 40
\pinlabel $l$ at 135 40
\endlabellist
\centering
\includegraphics[height=1.8cm]{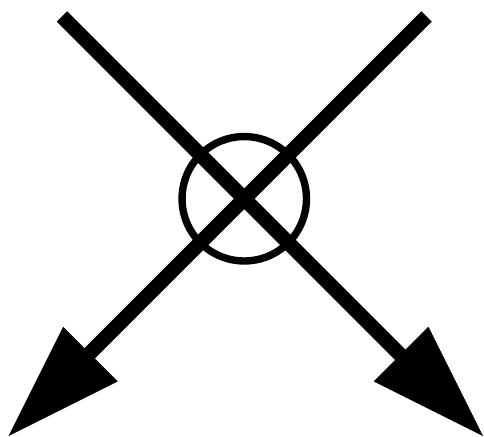}
\caption{A state assignment near a virtual crossing}
\label{fig:virtual_state}
\end{figure}
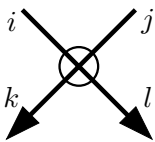

If we compute $\Phi_\pi(D)$ using the same bases, we have
$$
\Phi_\pi(D) = \sum_{\sigma} \omega(\sigma) = \text{Tr}(\phi) = \Phi'_\pi(D').
$$
This shows the definition of $\Phi'_\pi(D')$ does not depend on the choice of curves $c_q, q \in D \cap (\{0\} \times I)$. Since $\Phi_\pi'$ does not depend on a choice of basis, neither does $\Phi_\pi$.
\end{proof}

We write $\Phi$ rather than $\Phi_\pi$ in the next definition because we will ultimately show $\Phi$ does not depend on the choice of covering map $\pi$. Before proving this, however, we give the main result of the section.
\begin{defn}
\label{def:pseudo}
Let $L \subset T^2 \times I$ be a framed, oriented, $V$-colored link and $D$ a diagram for $L$ with framing parallel to $T^2$. Define the \emph{pseudo-operator invariant} of $L$, depending on $(\mathscr{A}, R)$ and $\mu$, by
$$
\Phi(L) = \Phi^{\mathscr{A}, R, \mu}(L) = \Phi^{\mathscr{A}, R, \mu}_\pi(D).
$$
\end{defn}

\begin{thm}
\label{thm:invariance}
$\Phi$ is an invariant of framed, oriented, $V$-colored links in $T^2 \times I$. That is, if $D_1$, $D_2$ are two diagrams of a framed, oriented, $V$-colored link $L \subset T^2 \times I$ with each having framing parallel to $T^2 \times \{0\}$, then $\Phi(D_1) = \Phi(D_2)$.
\end{thm}

\begin{proof}
Consider the lift $\tilde{D}_i$ of $D_i$ to $\R^2$ for $i = 1, 2$. By construction, $\tilde{D}_i$ is the diagram of a biperiodic link $\tilde{L} \subset \R^2 \times I$ such that the critical points of $\tilde{D}_i$ are lifts of critical points of $D_i$. Let $f_t(x) : I \times (T^2 \times I) \to T^2 \times I$ be an ambient isotopy carrying $D_1$ to $D_2$, so that $f_0 \equiv \id$ and $f_1(D_1) = D_2$. Then $f_t$ lifts to a biperiodic isotopy $\tilde{f}_t$ of $\R^2 \times I$ taking $\tilde{D}_1$ to $\tilde{D}_2$. Because $\tilde{D}_1$ and $\tilde{D}_2$ are locally blackboard-framed tangle diagrams, a well-known theorem \cite{rt90, fy89} asserts that $\tilde{f}_t$ decomposes into a sequence $\tilde{g}_t$ of diagram-preserving isotopies and the moves shown in Figure \ref{fig:tangle_moves} (with all possible orientations). We assume the isotopies and moves are biperiodic, i.e.~applied to each lifted copy of a region of $D_1$ simultaneously.

\begin{figure}[H]
\subcaptionbox{}{
\includegraphics[scale=0.5]{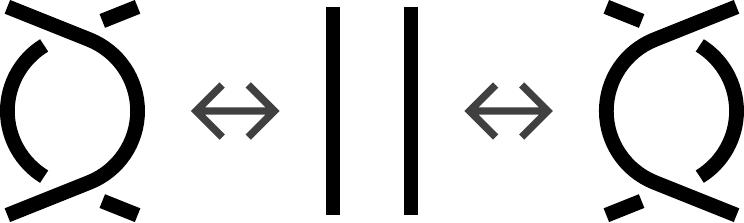}
}
\hspace{0.75cm}
\subcaptionbox{}{
\includegraphics[scale=0.5]{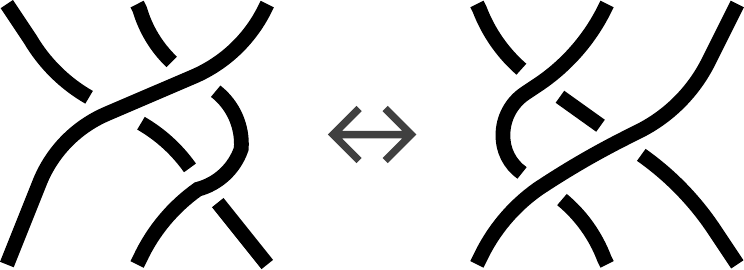}
}
\hspace{0.75cm}
\subcaptionbox{}{
\includegraphics[scale=0.5]{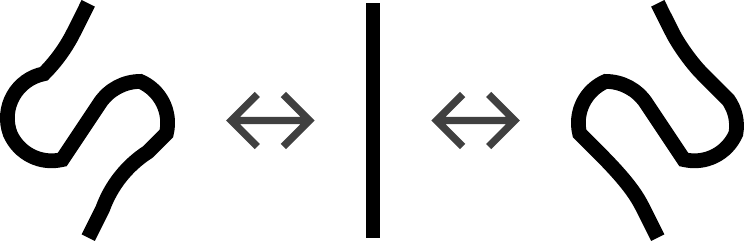}
}
\subcaptionbox{}{
\includegraphics[scale=0.5]{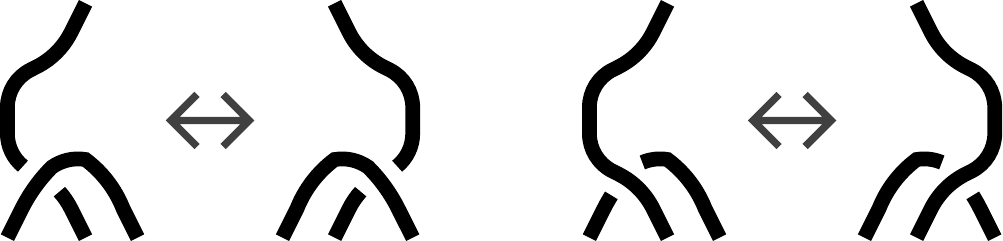}
}
\hspace{2cm}
\subcaptionbox{}{
\includegraphics[scale=0.5]{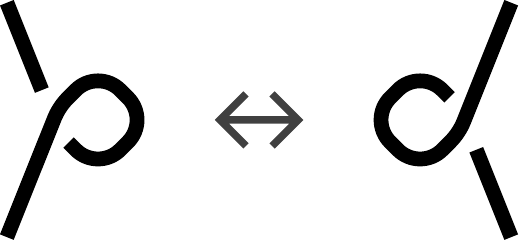}
}
\caption{Local Tangle Moves}
\label{fig:tangle_moves}
\end{figure}

Because $\tilde{g}_t$ is biperiodic, it descends to a sequence $g_t$ of the same moves on $T^2$ carrying $D_1$ to $D_2$. Hence it suffices to check invariance of $\Phi$ under each local move, which follows from properties of $\mathscr{A}$, $R$, and $\mu$. For example, the equation $R^{-1} \circ R = \id = R \circ R^{-1}$ implies invariance under move (a). Move (b) follows from the fact that $R$ satisfies the Yang-Baxter equation \cite{t88-2}, and moves (c)--(e) also follow from properties of $R$ and $\mu$---see \cite[Thm.~3.6]{km91} for details.
\end{proof}

\begin{rmk}
\label{rmk:virtual}
The construction of $\Phi_\pi$ given in the proof of Lemma \ref{thm:basis_invar} is similar to Kauffman's quantum invariant for virtual links \cite{k99}, in that virtual crossings are associated with the transposition map $\tau$. However, the two invariants have significant differences. We can think of the virtual diagram $D'$ in the proof of Lemma \ref{thm:basis_invar} as the diagram of a tangle on a cylinder $S^1 \times I$: the original diagram $D$ sits on the ``front'' of the cylinder, while the added curves $c_q$ circle around the ``back.'' This is one difference between our invariant and Kauffman's---the use of a cylinder to create the virtual diagram rather than a torus. Another difference is that Kauffman's invariant is defined in the context of \textit{rotational virtual knot theory} (see \cite{k15})---it is not invariant under virtual Reidemeister I-moves. We achieve invariance under virtual $I$-moves by placing all classical critical points on the front of the cylinder, where the orientation of the cylinder matches the orientation of the virtual diagram. If a critical point were moved to the back of the cylinder, that point's orientation on the cylinder would not match its orientation in the virtual diagram, and the two local operator assignments in the two constructions of $\Phi_\pi$ would disagree.
\end{rmk}

We use the phrase ``pseudo-operator invariant'' because, as remarked above, a torus diagram $D \subset T^2$ cannot generally be associated with an $\mathscr{A}$-linear operator using our construction. It is interesting that the local assignments of $\mathscr{A}$-linear operators to critical points of $D$ still allow us to define $\Phi(D)$, which seems to encode geometric information about the link $L$. For an example of computing $\Phi$ with a specific $\mathscr{A}$, see Section 6.

The fact that $\Phi(D) = \Phi_\pi(D)$ does not depend on $\pi$ follows from the proposition below, which shows $\Phi_\pi$ is invariant under orientation-preserving homeomorphisms of $T^2$.
\begin{prop}
\label{thm:dehn}
Let $D, D' \subset T^2$ be oriented, $V$-colored link diagrams with blackboard framing. If $f$ is an orientation-preserving homeomorphism of $T^2$ satisfying $f(D) = D'$, then $\Phi_\pi(D) = \Phi_\pi(D')$.
\end{prop}

\begin{proof}
Since $\Phi_\pi$ is an isotopy invariant, it suffices to prove the theorem for a set of homeomorphisms generating the mapping class group $\text{Mod}(T^2)$. To this end, we consider two Dehn twists, about two curves in $T^2$ which lift via $\pi$ to horizontal and vertical lines in $\R^2$. Let $l \subset T^2$ be a simple closed curve lifting to a vertical line in $\R^2$ such that $l$ contains no critical points of $D$, and choose a bicollar neighborhood $N(l)$ of $l$ satisfying the following conditions:
\begin{enumerate}[label=(\roman*)]
\item No critical points of $D$ occur within $N(l)$.
\item Each connected component of $D \cap N(l)$ intersects $l$ only once, transversely.
\end{enumerate}
Now suppose $f : T^2 \to T^2$ is an upward twist (from left to right) about $l$ which is the identity outside of $N(l)$. See Figure \ref{fig:twist} for an example. Let $c$ be a component of $D \cap N(l)$---then $c$ is a curve which increases or decreases monotonically as it travels across $N(l)$ from left to right. If $c$ is increasing, $f(c)$ is also monotonically increasing and contains no critical points. If $c$ is decreasing, $f(c)$ contains a minimum to the left of $l$ and a maximum to the right of $l$ and no critical points other than these (see Figure \ref{fig:twist}). The cases of twisting downward and twisting about horizontal lines are similar. Finally, because $f$ is injective, the crossing points of $D'$ are the same as those of $D$.

\begin{figure}[H]
\centering
\includegraphics[height=4cm]{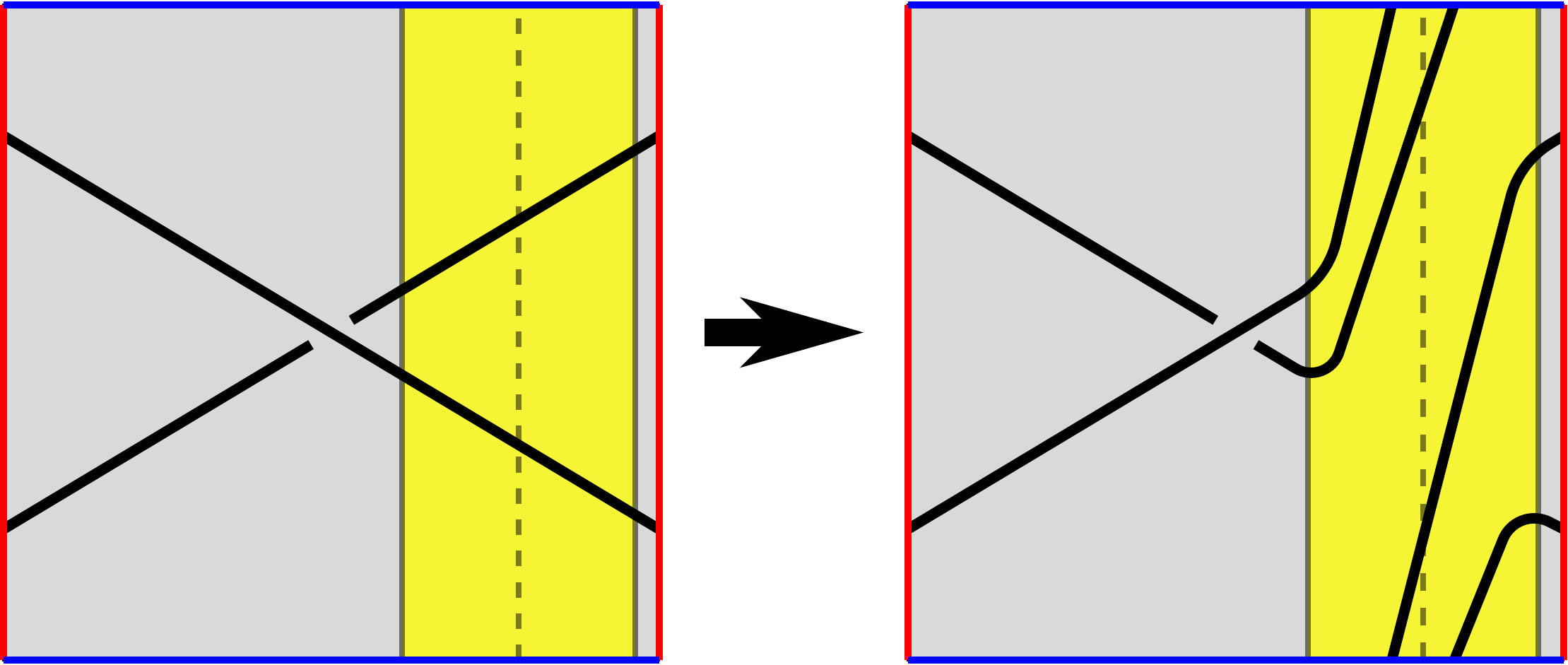}
\caption{Twisting a diagram}
\label{fig:twist}
\end{figure}

It follows that the only critical points of $f(D) = D'$ which do not occur in $D$ are max-min pairs formed as above. When $\Phi_\pi(D')$ is computed as a state sum, the weights of these max-min pairs cancel as in identity (c) of Figure \ref{fig:tangle_moves}. We conclude $\Phi_\pi(D) = \Phi_\pi(D')$.
\end{proof}

\begin{cor}
\label{thm:lift_choice}
The value $\Phi_\pi(D)$ does not depend on the choice of covering map $\pi : \R^2 \to T^2$.
\end{cor}
\begin{proof}
Let $\pi_1, \pi_2 : \R^2 \to T^2$ be two smooth, orientation-preserving covering maps with fundamental domain $I^2$. The uniqueness property of covering spaces gives an orientation-preserving homeomorphism $\tilde{f}: \R^2 \to \R^2$ which satisfies $\pi_2 \circ \tilde{f} = \pi_1$, and since $\pi_1$ and $\pi_2$ have the same fundamental domain and deck transformations, $\tilde{f}$ descends to an orientation-preserving homeomorphism $f$ of $T^2$ satisfying $f \circ \pi_1 = \pi_2$. By Proposition \ref{thm:dehn}, and noting the value $\Phi_{\pi_i}(D)$ is completely determined by the lift ${\pi_i}^{-1}(D)$,
$$
\Phi_{\pi_1}(D) = \Phi_{\pi_1}(f(D)) = \Phi_{f \circ \pi_1}(D) = \Phi_{\pi_2}(D).
$$
\end{proof}

We conclude the section with a general property of pseudo-operator invariants. Though the proposition is a simple observation, it motivates the skein theory to come in Section 5.

\begin{prop}
\label{thm:why_two}
Let $K \subset T^2 \times I$ be a knot projecting to an essential, simple closed curve in $T^2 \times I$ with framing parallel to $T^2 \times \{0\}$. If $K$ is colored by an $n$-dimensional $\mathscr{A}$-module $V$,
$$
\Phi(K) = n.
$$
\end{prop}

\begin{proof}
Applying Proposition \ref{thm:dehn}, we may assume without loss of generality that $K$ lifts to a vertical line in $\R^2$. Then the proof of Lemma \ref{thm:basis_invar} shows $\Phi(K)$ is the trace of the identity map of $V$, so $\Phi(K) = \dim(V) = n$.
\end{proof}

\section{Quantum Invariants for $sl(2, \C)$ and the Toroidal Colored Jones Polynomial}

\subsection{An Invariant of Framed, Unoriented Links in $T^2 \times I$}

As with the colored Jones polynomial, we now specialize to $\mathscr{A} = \mathscr{A}_q = \mathcal{U}_q(sl(2, \C))$, $q = e^{2\pi i/r}$, and limit $\mathscr{A}$-modules to the set $\{V^1, \dots, V^r\}$ as in Section 2.2. For a link $L \subset T^2 \times I$ with $k$ components and multi-integer ${\bf n} = (n_1, \dots, n_k)$, let $L({\bf n}) = L(n_1, \dots, n_k)$ be the $V$-colored link $L$ with $j$th component colored by $V^{n_j}$. By $1 \leq {\bf n} \leq r$ in the definition below we mean that $1 \leq n_j \leq r$ for all $n_j$.

\begin{defn}
Given a framed, unoriented link $L \subset T^2 \times I$ with $k$ components, fix an orientation of each component. For $q = e^{2\pi i/r}$, $1 \leq {\bf n} \leq r$, define $\hat{J}^T_{{\bf n}, q}(L)$ by
$$
\hat{J}^T_{{\bf n}, q}(L) = \Phi^{\mathscr{A}_q, R, \mu}(L({\bf n})).
$$
\end{defn}

\begin{thm}
\label{thm:orientation}
$\hat{J}^T_{{\bf n}, q}(L)$ is an invariant of framed, unoriented links in $T^2 \times I$. That is, $\hat{J}^T_{{\bf n}, q}(L)$ does not depend on the orientation chosen for each component of $L$.
\end{thm}

\begin{proof}
Let $D, D' \subset T^2$ be oriented diagrams of $L({\bf n})$ with framing parallel to $T^2 \times \{0\}$ and $D'$ obtained from $D$ by changing the orientation of a link component $C$. It suffices to show
$$
\Phi^{\mathscr{A}_q, R, \mu}(D) = \Phi^{\mathscr{A}_q, R, \mu}(D').
$$

Suppose $C$ is colored by $V^n$ and let $C'$ be the corresponding component of $D'$, also colored by $V^n$. Let $p \in C$ be a critical point, $p'$ the same point of $C'$, and $T(p)$, $T(p')$ small rectangular neighborhoods of each. Then each copy of $V^n$ coming from $C$ in $T(p)_\pm$ corresponds to a copy of $(V^n)^*$ in $T(p')_\pm$ and vice versa. $V^n$ is self-dual as an $\mathscr{A}$-module via a canonical isomorphism $\varphi : (V^n)^* \to V^n$, and we use $\varphi$ to identify the modules ${T(p)}_\pm$ and $T(p')_\pm$.

We apply Lemma 3.18 and Remark 3.26 of \cite{km91}, which state the following: if $n$ is odd, $\mathcal{F}_{T(p)} = \mathcal{F}_{T(p')}$ as maps from ${T(p)}_-$ to ${T(p)}_+$ for all $p$. Thus $\Phi(D) = \Phi(D')$ if $n$ is odd. Suppose $n$ is even. If $p$ is a crossing then $\mathcal{F}_{T(p)} = \mathcal{F}_{T(p')}$, and if $p$ is an extreme point then $\mathcal{F}_{T(p)} = -\mathcal{F}_{T(p')}$. Self-duality of $V^n$ induces a bijection between the states of $D$ and the states of $D'$, and it follows that if $\sigma$ is a state of $D$ and $\sigma'$ the corresponding state of $D'$, $\omega(\sigma) = (-1)^j \omega(\sigma')$, where $j$ is the total number of extreme points of $C$. Since $C \subset T^2$ is a closed curve, $j$ is even and $\omega(\sigma) = \omega(\sigma')$. We conclude $\Phi(D) = \Phi(D')$ if $n$ is even.
\end{proof}

The invariant $\hat{J}^T_{{\bf n}, q}$ should be thought of as a toroidal analogue of the invariant $J_{L, {\bf k}}$ of \cite{km91}. One might also be reminded of the Kauffman bracket skein module, another invariant of framed, unoriented links---this comparison will be made precise in the next section. Like the invariant $J_{L, {\bf k}}$ of \cite{km91} or the Kauffman bracket skein module of $S^3$, $\hat{J}^T_{{\bf n}, q}$ can be normalized to obtain an invariant of oriented, unframed links in $T^2 \times I$ analogous to the colored Jones polynomial.

\subsection{The Toroidal Colored Jones Polynomial}

To create an invariant of unframed links, we use the fact that as an endomorphism of $V^n \otimes V^n$, $0 \leq n \leq r$, the $R$-matrix $R$ satisfies \cite{km91, t88}
$$
(\id \otimes E_\mu)(R^{\pm 1} \otimes \id)(\id \otimes N) = q^{\pm\frac{(n^2 - 1)}{4}} \id
$$
where $N : \C \to V^n \otimes (V^n)^*$ and $E_\mu : V^n \otimes (V^n)^* \to \C$ are the maps given in Theorem \ref{thm:fund_tangles}. Pictorially, this is equivalent to the two equations
\begin{align}
\label{eq:kinks}
\hat{J}^T_{{\bf n}, q}(\dpcurl) = q^{(n^2 - 1)/4} \cdot \hat{J}^T_{{\bf n}, q}(\dstrand) & & \hat{J}^T_{{\bf n}, q}(\dncurl) = q^{-(n^2 - 1)/4}\cdot \hat{J}^T_{{\bf n}, q}(\dstrand)
\end{align}
where the diagrams represent oriented, unframed links and the component shown is colored by $V^n$.

If $L \subset T^2 \times I$ is oriented and unframed, any two diagrams of $L$ are related by a sequence of the moves in Figure \ref{fig:tangle_moves} and additions or removals of curls as shown in (\ref{eq:reid_1}) below.
\begin{equation}
\label{eq:reid_1}
\dpcurl \leftrightarrow \dstrand \leftrightarrow \dncurl
\end{equation}
This fact, combined with (\ref{eq:kinks}), motivates Definition \ref{def:complicated_cjp}. Similar to above, given a diagram $D \subset T^2$ of a $k$-component link and multi-integer ${\bf n} = (n_1, \dots, n_k)$, let $D({\bf n})$ indicate $D$ with $j$th component colored by $V^{n_j}$. Define $\hat{J}^T_{{\bf n}, q}(D) := \Phi^{\mathscr{A}_q, R, \mu}(D({\bf n}))$.

\begin{defn}
\label{def:complicated_cjp}
Let $L \subset T^2 \times I$ be an oriented, unframed link with $k$ components, $D \subset T^2$ a diagram of $L$, and $C_1, \dots, C_k$ the link components of $D$. Define $J^T_{{\bf n},q}(L)$ by
$$
J^T_{{\bf n}, q}(L) = q^{\alpha/4} \hat{J}^T_{{\bf n}, q}(D)
$$
where $\alpha = -\sum_{j = 1}^k w(C_j) \cdot ({n_j}^2 - 1)$ and $w(C_j)$ is the writhe of $C_j$, i.e.~the sum of the signs of its self-crossings.
\end{defn}

It follows from (\ref{eq:kinks}) and the proof of Theorem \ref{thm:invariance} that $J^T_{{\bf n}, q}(L)$ is an invariant of oriented, unframed links in $T^2 \times I$.

If all components of $L$ are given the same color, i.e.~if ${\bf n} = (n, \dots, n)$ for some $n \in \N$, we can define a similar invariant which agrees with $J^T_{{\bf n}, q}$ if $L$ is a knot. This next definition is our analogue of the colored Jones polynomial.

\begin{defn}
\label{def:simple_cjp}
For an oriented, unframed link $L \subset T^2 \times I$ with diagram $D$ and $n \in \N$, Define the \emph{$n$th toroidal colored Jones polynomial} $J^T_n(L; q)$ of $L$ by
$$
J^T_n(L;q) = (q^{(n^2 - 1)/4})^{-w(D)} \hat{J}^T_{(n, \dots, n), q}(D),
$$
where $w(D)$ is the writhe of $D$.
\end{defn}

Compare Definition \ref{def:simple_cjp} with (\ref{eq:tangle_jones})---the definitions are analogous except for a factor of $1/[n]$. This factor may be included in the definition of $J_n$ because $J_n$ is always divisible by $[n]$ as a Laurent polynomial in $q$; this follows from the $\mathscr{A}$-linearity of the operator in Theorem \ref{thm:fund_tangles}. Because there is no guarantee of global $\mathscr{A}$-linearity in the construction of $J^T_n$, we cannot divide by $[n]$. In particular, if $q = e^{2\pi i/n}$,
\begin{equation}
\label{eq:zero_n}
[n] = \frac{e^{\pi i} - e^{-\pi i}}{\{1\}} = 0.
\end{equation}

Additionally, the root of unity $q$ may be replaced by an indeterminate in the definition of $J^T_n(L; q)$ without affecting calculations---see, for example, \cite{my18}. This justifies thinking of $J^T_n(L; q)$ as a Laurent polynomial in $q$ and accounts for our slight change in notation. We prefer $J^T_n$ to $J^T_{{\bf n}, q}$ for simplicity in calculations, and don't make use of $J^T_{{\bf n}, q}$ outside of this section. We also remark that, because Dehn twists do not affect the signs of crossings in a diagram, Proposition \ref{thm:dehn} extends to $J^T_n$ and $J^T_{{\bf n}, q}$:

\begin{prop}
\label{thm:follow_up}
Let $L, L' \subset T^2 \times I$ be oriented links with respective diagrams $D, D' \subset T^2$. If $f$ is an orientation-preserving homeomorphism of $T^2$ satisfying $f(D) = D'$, then $J^T_n(L;q) = J^T_n(L';q)$ and $J^T_{{\bf n}, q}(L) = J^T_{{\bf n}, q}(L')$ for all $n$ and ${\bf n}$.
\end{prop}

As a final result of the subsection, we give a cabling formula for $\hat{J}^T_{{\bf n},q}$ analogous to the cabling formula for the invariant $J_{L, {\bf k}}$ (see \cite[Thm.~4.15]{km91}). For a framed, unoriented link $L \subset T^2 \times I$, this expresses the value $\hat{J}^T_{{\bf n},q}(L)$ in terms of $\hat{J}^T_{{\bf 2},q} = \hat{J}^T_{(2,\dots, 2),q}$ evaluated on certain cablings of $L$. In the next section, this will allow us to develop $\hat{J}^T_{{\bf n},q}$ from a skein-theoretic viewpoint.

Let $L$ be a $k$-component link and ${\bf n} = (n_1, \dots, n_k)$ a multi-integer. As in Section 2.1, denote by $L^{\bf n}$ the cabling of $L$ which replaces the $j$th component of $L$ by $n_j$ parallel push-offs of itself, oriented compatibly if the link is oriented, with associated diagram $D^{\bf n}$. Below, the sum is over all ${\bf i} = (i_1, \dots, i_k)$ with $1 \leq i_j \leq (n_j - 1)/2$, and $(-1)^{\bf i} \binom{{\bf n} - 1 - {\bf i}}{\bf i} = \prod_{j = 1}^k (-1)^{i_j} \binom{n_j - 1 - i_j}{i_j}$.

\begin{thm}[Cabling Formula]
\label{thm:cabling}
Let $L \subset (T^2 \times I)$ be a framed, unoriented link, ${\bf n}$ a coloring of $L$, and $D$ a torus diagram for $L$ with framing parallel to $T^2 \times \{0\}$. Then
\begin{align*}
\hat{J}^T_{{\bf n},q}(L) &= \sum_{{\bf i} = 0}^{({\bf n} - 1)/2} (-1)^{\bf i} \binom{{\bf n} - 1 - {\bf i}}{\bf i}\hat{J}^T_{{\bf 2},q}(L^{{\bf n} - 1 - 2{\bf i}}) \\
&=  \sum_{{\bf i} = 0}^{({\bf n} - 1)/2} (-1)^{\bf i} \binom{{\bf n} - 1 - {\bf i}}{\bf i} q^{3 w(D^{{\bf n} - 1 - 2{\bf i}})/4}J^T_2(L^{{\bf n} - 1 - 2{\bf i}}; q).
\end{align*}
\end{thm}

We sketch the proof, following closely the proof of Theorem 4.15 in \cite{km91}. We first require a lemma (c.f. \cite[Lem.~3.10]{km91}), which gives useful properties of pseudo-operator invariants.

\begin{lemma}
\label{thm:operator_cabling}
Let $(\mathscr{A}, R)$ be a quasitriangular Hopf algebra with good unit $\mu$. Let $D \subset T^2 \times I$ be a colored torus diagram and $C$ a link component of $D$ colored by $V$. 
\begin{enumerate}[label=(\alph*)]
\item If $V = X \oplus Y$, or more generally $V$ is an extension of $Y$ by $X$ (i.e.~there is a short exact sequence $0 \to X \to V \to Y \to 0$ of $\mathscr{A}$-modules), then
$$
\Phi(D) = \Phi(D_X) + \Phi(D_Y)
$$
where $D_Z$ denotes the torus diagram obtained by changing the color of $C$ to $Z$.

\item If $V = X \otimes Y$, then
$$
\Phi(D) = \Phi(D_{XY})
$$
where $D_{XY}$ is the diagram obtained by replacing $C$ by two parallel pushoffs of itself (using the framing) colored by $X$ and $Y$, respectively.
\end{enumerate}
\end{lemma}

\begin{proof}
To prove (a), fix bases $B_V$, $B_X$, and $B_Y$ so that $B_X \subset B_V$ (viewing $X$ as a subspace of $V$) and $B_Y$ is the projection of $\tilde{B}_Y = B_V - B_X$. We call state labels from $B_X$ or $B_{X^*}$ $X$-{\it labels}, whereas those from $\tilde{B}_Y$ or $\tilde{B}_{Y^*}$ are $Y$-{\it labels}.

If $\sigma$ is a state of $D$ with non-zero weight, then the corresponding labels on the arcs of $C$ must be either all $X$-labels (written $\sigma | C \subset X$) or all $Y$-labels (written $\sigma | C \subset Y$). This follows from the $\mathscr{A}$-invariance of $X \subset V$ (and dually of $Y^* \subset V^*$)---if the component of $C$ on one side of a critical point has an $X$-label and the component on the other side has a $Y$-label, the weight of the critical point in that state will zero. From this, we see
$$
\Phi(D) = \sum \omega(\sigma) = \sum_{\sigma | C \subset X} \omega(\sigma) + \sum_{\sigma | C \subset Y} \omega(\sigma) = \Phi(D_X) + \Phi(D_Y)
$$
as desired.

Statement (b) is a fundamental property of operator invariants---see \cite[Ch.~1]{t94} or \cite[Lem.~3.10]{km91}. As with statement (a), we extend to the case of pseudo-operator invariants by considering each critical point individually.
\end{proof}

\begin{proof}[Proof of Theorem \protect\ref{thm:cabling}]

It is a classical result (see \cite[Cor.~2.15]{km91}) that, for $0 \leq n < r$, the equality
\begin{equation}
\label{eq:rep_ring}
V^{n + 1} = \sum_{j = 0}^{n/2} (-1)^j \binom{n - j}{j} (V^2)^{n - 2j}
\end{equation}
holds in the representation ring of $\mathscr{A}_q$, where $q = e^{2\pi i/r}$ and the sum is over all $j$ with $0 \leq 2j \leq n$. Here $V^n$ refers to the $\mathscr{A}_q$-module $V^n$, while $(V^n)^j$ indicates the $j$th tensor product of $V^n$ with itself. The proof of (\ref{eq:rep_ring}) uses the fact that the modules $V^n$ satisfy the recurrence relation $V^{n + 1} = V^2 V^n - V^{n - 1}$, the same recurrence relation defining the Chebyshev polynomials in (\ref{eq:chebyshev}). 

The first equality of Theorem \ref{thm:cabling} now follows from combining Lemma \ref{thm:operator_cabling} and (\ref{eq:rep_ring}), and the second equality comes from Definition \ref{def:simple_cjp}.

\end{proof}

\begin{rmk}
Toroidal analogues of other link invariants can be constructed by considering quantum groups other than $U_q(sl(2,\C))$. For example, letting $\mathscr{A} = U_q(sl(m, \C))$ for general $m$ gives a toroidal analogue of the specialization $P_L(q^m, q - q^{-1})$, where $P_L$ is the two-variable HOMFLY polynomial of a link $L$. Letting $\mathscr{A} = U_qG$, where $G = so(m)$ or $sp(2m)$, leads to a toroidal analogue of a certain specialization of the two-variable Kauffman polynomial, depending on the choice of $G$. See \cite[Sec.~6.1]{rt90}. In Section 5, we'll construct a toroidal analogue of the $SU(2)$ colored Jones polynomial by setting $\mathscr{A} = SU(2)_q$.
\end{rmk}

\section{The Toroidal Colored Jones Polynomial and Skein Theory}

In this section only, we consider the invariant defined by specializing $\Phi$ to the quantum group $SU(2)_q$ rather than $U_q(sl(2, \C))$. This constitutes a certain normalization of the $sl(2, \C)$ invariant and we denote the $SU(2)$ version by the same notation, $\hat{J}^T_{{\bf n}, q}$. This is consistent with literature on the colored Jones polynomial and the operators involved are discussed, for example, in \cite{m93, kr89}.

The goal of the section is to develop the $SU(2)$ toroidal colored Jones polynomial skein-theoretically. This begins with an observation about the level two framed invariant $\hat{J}_{{\bf 2},q}$.

\begin{lemma}
\label{thm:pre_kauffman}
The level two $SU(2)$ invariant $\hat{J}^T_{{\bf 2},q}$ has the following properties:
\begin{enumerate}[label=(\alph*)]
\item $\hat{J}^T_{{\bf 2},q}(\varnothing) = 1$.
\item Let $C \subset T^2$ be a simple closed curve disjoint from a diagram $D \subset T^2$. 
\begin{enumerate}[label=(\roman*)]
\item If $C$ is contractible, $\hat{J}^T_{{\bf 2},q}(C \sqcup D) = (-q^{1/2} - q^{-1/2}) \hat{J}^T_{{\bf 2},q}(D) = -[2] \hat{J}^T_{{\bf 2},q}(D)$.
\item If $C$ is not contractible, $\hat{J}^T_{{\bf 2},q}(C \sqcup D) = 2 \hat{J}^T_{{\bf 2},q}(D)$.
\end{enumerate}
\item $\hat{J}^T_{{\bf 2},q}(\pcross) = q^{1/4} \hat{J}^T_{{\bf 2},q}(\astate) + q^{-1/4} \hat{J}^T_{{\bf 2},q}(\bstate)$.
\end{enumerate}
\end{lemma}

\begin{proof}
Properties (a), (b)(i) and (c) are identical to the relations defining the usual Kauffman bracket. Because they are local properties ((b)(i) is local in the sense that we can assume $C$ exists in a coordinate neighborhood of $T^2$), the proofs are the same as for the usual $SU(2)$ Jones polynomial---see \cite[Thm.~4.1]{m93}. Each property reduces to an algebraic statement about the quantum group $SU(2)_q$.

To prove property (b)(ii) holds for $\hat{J}^T_{{\bf 2}, q}$, suppose $C \subset T^2$ is a simple, closed essential curve. Then $\hat{J}^T_{2,q}(C) = 2$ by Proposition \ref{thm:why_two}. The general statement follows from the multiplicativity of $\hat{J}_{{\emph n}, q}$ on disjoint diagrams; that is,
$$
\hat{J}^T_{{\bf 2}, q}(U \sqcup D) = \hat{J}^T_{{\bf 2}, q}(U) \cdot \hat{J}^T_{{\bf 2}, q}(D) = 2 \hat{J}^T_{{\bf 2}, q}(D).
$$
\end{proof}

Lemma \ref{thm:pre_kauffman} leads to the following definition and theorem:

\begin{defn}
\label{def:kauffman}
Define a Kauffman-type bracket $\langle * \rangle_\tau \in \Z[A^{\pm 1}, z]$ on link diagrams in $T^2$ (and framed links in $T^2 \times I$) by the relations
\begin{enumerate}[label=(\alph*)]
\item $\langle \varnothing \rangle_\tau= 1$.
\item Let $C \subset T^2$ be a simple closed curve disjoint from a diagram $D \subset T^2$. 
\begin{enumerate}[label=(\roman*)]
\item If $C$ is contractible, $\langle C \sqcup D \rangle_\tau = (-A^2 - A^{-2}) \langle D \rangle_\tau$.
\item If $C$ is not contractible, $\langle C \sqcup D \rangle_\tau = z \cdot \langle D \rangle_\tau$.
\end{enumerate}
\item $\langle \pcross \rangle_\tau = A \langle \astate \rangle_\tau + A^{-1} \langle \bstate \rangle_\tau$.
\end{enumerate}
\end{defn}

\begin{thm}
\label{thm:kauffman}
For any framed link $L \subset T^2 \times I$,
$$
\hat{J}^T_{{\bf 2},q}(L) = \langle L \rangle_\tau |_{A^4 = q, z = 2}.
$$
\end{thm}

To extend Theorem \ref{thm:kauffman} to all values of ${\bf n}$ for $\hat{J}_{{\bf n}, q}$, we consider the skein module of the thickened torus, $\mathscr{S}(T^2)$. As Section 2.1 discusses, a basis for $\mathscr{S}(T^2)$ as a $\Z[A^{\pm 1}]$-module is given by positive powers of the tuples $(a,b)$ such that either $a = b = 0$ or $a$ and $b$ are coprime. Let $p_2 : \mathscr{S}(T^2) \to \Z[A^{\pm1}]$ be the $\Z[A^{\pm 1}]$-linear map defined by:
$$
p_2((a,b)^m) = 
\begin{cases}
1 & a = b = 0 \\
2^m & \text{otherwise}
\end{cases}.
$$
Then it's clear that, for any framed link $L \subset T^2 \times I$,
$$
\langle L \rangle_\tau |_{z = 2} = p_2(\langle L \rangle_T)
$$
where $\langle L \rangle_T$ is the class of $L$ in $\mathscr{S}(T^2)$.

We can now state the full result using $\mathscr{S}(T^2)$.
\begin{thm}
\label{thm:skein_op}
For any oriented, unframed link $L \subset T^2 \times I$ with diagram $D \subset T^2$,
$$
\hat{J}^T_{{\bf n}, q}(L) = p_2(\langle S_{n_1 - 1}(z), \dots, S_{n_k - 1}(z) \rangle_{T,D}) \big|_{A^4 = q}
$$
where $\hat{J}^T_{{\bf n}, q}$ is the $SU(2)$ invariant and $\langle * \rangle_{T,D} : \mathscr{S}(\mathcal{A})^{\otimes k} \to \mathscr{S}(T^2)$ is the multi-bracket of equation (\ref{eq:torus_multibracket}).
\end{thm}

\begin{proof}
A closed formula for the $n$th Chebyshev polynomial, as defined in (\ref{eq:chebyshev}), is given by
\begin{equation}
\label{eq:two}
S_n(z) = \sum_{j = 0}^{n/2} (-1)^j \binom{n - j}{j} z^{n - 2j}
\end{equation}
where the sum is over all integers $j$ with $0 \leq 2j \leq n$. Let $L \subset T^2 \times I$ be a $k$-component link and ${\bf n} = (n_1, \dots, n_k)$ a multi-integer. Applying the above formula and the multilinearity of $p_2$ and the Kauffman multi-bracket, we have
\begin{align*}
p_2(\langle S_{n_1}(z), \dots, S_{n_k}(z)\rangle_{T,D}) &= \sum_{{\bf j} = 0}^{{\bf n}/2} (-1)^{\bf j} \binom{{\bf n} - {\bf j}}{\bf j} p_2(\langle L^{{\bf n} - 2{\bf j}}\rangle_T) \\
&= \sum_{{\bf j} = 0}^{{\bf n}/2} (-1)^{\bf j} \binom{{\bf n} - {\bf j}}{\bf j} \hat{J}^T_{{\bf 2},q} (L^{{\bf n} - 2{\bf j}}) \\
&= \hat{J}^T_{{\bf n} - 1, q}(L).
\end{align*}
The second equality is Theorem \ref{thm:kauffman}. The third comes from Theorem \ref{thm:cabling}, which applies in the $SU(2)$ theory since the same relation $V^{j + 1} = V^2 V^j - V^{j - 1}$ holds in the representation ring.
\end{proof}

Having constructed the $SU(2)$ $\hat{J}_{{\bf n}, q}$ skein-theoretically, we can define a skein-theoretic toroidal colored Jones polynomial. The $SU(2)$ version of $\hat{J}^T_{{\bf n}, q}$ satisfies \cite{ms91}
\begin{align*}
\hat{J}^T_{{\bf n}, q}(\dpcurl) &= (-1)^{n - 1} q^{(n^2 - 1)/4} \cdot \hat{J}^T_{{\bf n}, q}(\dstrand) \\
\hat{J}^T_{{\bf n}, q}(\dncurl) &= (-1)^{n - 1} q^{-(n^2 - 1)/4} \cdot \hat{J}^T_{{\bf n}, q}(\dstrand),
\end{align*}
where the strand shown in the diagram is colored by $V^n$. Subsequently:
\begin{defn}
The \emph{$SU(2)$ toroidal colored Jones polynomial} $J_n^T$ of an oriented, unframed link $L \subset T^2 \times I$ with diagram $D$ is defined by
$$
J^T_n(L;q) = ((-1)^{n - 1} q^{(n^2 - 1)/4})^{-w(D)} \hat{J}^T_{(n, \dots, n), q}(D).
$$
\end{defn}

We immediately have:

\begin{cor}
\label{thm:skein_cjp}
The $SU(2)$ toroidal colored Jones polynomial is defined skein-theoretically by
$$
J^T_n(L;q) = \Big[ ((-1)^{n - 1} A^{n^2 - 1})^{-w(D)} p_2(\langle S_{n - 1}(z), \dots, S_{n - 1}(z) \rangle_D) \Big]\Big|_{A^4 = q}.
$$
\end{cor}
Compare the right side of Theorem \ref{thm:skein_op} with equation (\ref{eq:skein_jones})---the missing factor of $1/(-A^2 - A^{-2})$ is analogous to the missing $1/[n]$ factor in the $U_q(sl(2,\C))$ case.

The skein theoretic definitions of Theorem \ref{thm:skein_op} and Corollary \ref{thm:skein_cjp} let us extend the $SU(2)$ invariants $\hat{J}^T_{{\bf n}, q}$ and $J^T_n$ to links in orientable manifolds other than $T^2 \times I$, using the bracket $\langle * \rangle_\tau$ of Definition \ref{def:kauffman} (with $z = 2$) as a generalized Kauffman bracket. In $S^3$ the bracket $\langle * \rangle_\tau$ coincides with the usual Kauffman bracket, and thus the invariant $\hat{J}^T_{{\bf n}, q}$ defined in $S^3$ is exactly the invariant $J_L$ of \cite{m93}. The $SU(2)$ toroidal colored Jones polynomial, defined skein-theoretically in $S^3$, satisfies
$$
J^T_n(L;q) = (-A^2 - A^{-2})|_{A^4 = q} \cdot J_n(L;q)
$$
where $J_n$ is the $SU(2)$ colored Jones polynomial of equation (\ref{eq:skein_jones}).

\section{The Volume Conjecture for the 2-by-2 Square Weave}

We now prove the volume conjecture for the \emph{$2$-by-$2$ square weave} $W \subset T^2 \times I$, the link shown in Figure \ref{fig:square_weave}. More generally, a diagram for the \emph{$2k$-by-$2l$ square weave} $W_{2k, 2l}$, $k,l \in \N$, is made by tiling $W$ to form a rectangular grid with $2k$ rows and $2l$ columns of crossings. We consider only even dimensions to ensure the diagram is alternating on the torus.

The complement of $W_{2k,2l}$ in $T^2 \times I$ is geometrically simple---Champanerkar, Kofman and Purcell \cite{ckp16} describe a complete hyperbolic structure for $(T^2 \times I) \setminus W_{2k,2l}$ consisting of $4kl$ regular ideal hyperbolic octahedra, one for each crossing. Thus
\begin{equation}
\label{eq:vol}
\rm{Vol}(W_{2k, 2l}) = 4kl \cdot v_\text{oct}.
\end{equation}
Separately, we have the following proposition.
\begin{prop}
\label{thm:gl}
Let $q = e^{2\pi i/n}$. If $L \subset T^2 \times I$ is an oriented link which has a diagram $D \subset T^2$ with $c$ crossings,
$$
\lim_{n \to \infty} \frac{2\pi}{n} \log | J^T_n(L;q)| \leq c \cdot v_\text{oct}.
$$
\end{prop}
\begin{proof}
The result follows from work of Garoufalidis and L{\^e} \cite[Cor.~8.10]{gl11} (see also Thm. 1.13). Recall the definition of $J_n^T$ as a state sum (Definition \ref{def:simple_cjp} and equations (\ref{eq:state_sum_1}) and (\ref{eq:state_sum_2})):
$$
J_n^T(L;q) = (q^{(n^2 - 1)/4})^{-w(D)}\sum_{\sigma}\prod_{p \in P} \omega_p(\sigma),
$$
where $P$ is the set of critical points of $D$ and $\sigma$ is a state of $D$. Since $q$ is a root of unity, $|J_n^T| = |\sum_{\sigma}\prod_{p \in P} \omega_p(\sigma)|$.

Let $m$ be the number of connected components of $D \setminus P$; then $D$ has $n^m$ states. If $\sigma'$ is a state with maximum modulus weight, then
\begin{equation}
\label{eq:gl_proof}
\log|J_n^T| = \log|\sum_{\sigma} \prod_{p \in P} \omega_p(\sigma)|| \leq \log|n^m \prod_{p \in P} \omega_p(\sigma')| = m\log n + \sum_{p \in P} \log |\omega_p(\sigma')|.
\end{equation}
Fix the standard basis for $V^n \cong \C^n$. In this basis, if $p \in P$ is an extremum then $|\omega_p(\sigma')| = 1$. If $p$ is a crossing then $\omega_p(\sigma')$ is an element of the $R$-matrix of $\mathscr{A}_q$ in the given basis, and Garoufalidis and L\^e proved
\begin{equation}
\label{eq:gl_bound}
\frac{2\pi}{n} \lim_{n \to \infty} \log |\omega_p(\sigma')| \leq v_\text{oct}.
\end{equation}
Considering $\lim_{n \to \infty} \frac{2\pi}{n} \log|J_n^T|$, the summands $\log|\omega_p(\sigma')|$ with $p$ a crossing are the only summands of (\ref{eq:gl_proof}) which don't vanish asymptotically. Thus, applying (\ref{eq:gl_bound}),
$$
\lim_{n \to \infty} \frac{2\pi}{n} \log|J_n^T| \leq c \cdot v_\text{oct}.
$$
\end{proof}

Equation (\ref{eq:vol}) and Proposition \ref{thm:gl} together give
\begin{equation}
\label{eq:bound}
\lim_{n \to \infty} \frac{2\pi}{n} \log | J^T_n(W_{2k, 2l};e^\frac{2\pi i}{n})| \leq 4kl \cdot v_\text{oct} = \rm{Vol}((T^2 \times I) \setminus W_{2k, 2l}).
\end{equation}
In other words, for the rectangular weave $W_{2k, 2l}$, the asymptotic growth of the toroidal colored Jones polynomial is bounded above by the volume of the complement. This makes the $2k$-by-$2l$ rectangular weave a natural object of study for our Volume Conjecture \ref{thm:conj}---to prove the conjecture for a link in this family, we need only show the upper bound in (\ref{eq:bound}) is achieved.

Before proving Theorem \ref{thm:main_one}, we give a formula for $J^T_n(W; q) = J^T_n(W_{2,2}; q)$. As an isomorphism of $V^n \otimes V^n$ the $R$-matrix $R$ of $\mathscr{A}_q$ is defined by weights $R^{ij}_{kl} \in \C$,
$$
R(e_k \otimes e_l) = \sum_{i,j = 0}^{n - 1} R^{ij}_{kl} e_i \otimes e_j,
$$
where $\{e_0, \dots, e_{n - 1}\}$ is a preferred basis of $V^{n}$. We have
\begin{equation}
\label{eq:r}
R^{ij}_{kl} = \sum_{m = 0}^{\text{min}(n - 1 - i, j)} \delta^l_{i + m} \delta^k_{j - m} \cdot q^\alpha \cdot \frac{\{l\}!\{n - 1 - k\}!}{\{i\}!\{m\}!\{n - 1 - j\}!}
\end{equation}
where
$$
\alpha = (i - (n - 1)/2)(j - (n - 1)/2) - m(i - j)/2 - m(m + 1)/4,
$$
$\delta$ is the Kronecker delta, and $\{m\}! = \{m\}\{m - 1\} \cdots \{2\}\{1\}$. Similarly, $R^{-1} : V^n \otimes V^n \to V^n \otimes V^n$ is defined by the scalars
\begin{equation}
\label{eq:r_inv}
(R^{-1})^{ij}_{kl} = \sum_{m = 0}^{\text{min}(n - 1 - i, j)} \delta^l_{i - m} \delta^k_{j + m} \cdot (-1)^m \cdot q^\beta \cdot \frac{\{k\}!\{n - 1 - l\}!}{\{j\}!\{m\}!\{n - 1 - i\}!}
\end{equation}
where
$$
\beta = -(i - (n - 1)/2)(j - (n - 1)/2) - m(i - j)/2 + m(m + 1)/4.
$$
Let $D$ be the diagram for $W$ in Figure \ref{fig:sq_weave_states}a, with components labelled by basis elements as shown---implicitly we've chosen a covering map where crossings are oriented downward and there are no maxima or minima. We have
$$
J^T_n(W; q) = \sum_{a,b,\dots, g, h = 0}^{n - 1} (R^{-1})^{fa}_{gb}R^{be}_{cf}R^{dg}_{ah}(R^{-1})^{hc}_{ed}.
$$

\begin{figure}[H]
\centering
\subcaptionbox{}[.48\textwidth]{
\labellist
\small\hair 2pt
\pinlabel $a$ at 245 365
\pinlabel $b$ at 245 260
\pinlabel $c$ at 245 160
\pinlabel $d$ at 245 95
\pinlabel $e$ at 185 95
\pinlabel $f$ at 185 365
\pinlabel $g$ at 185 260
\pinlabel $h$ at 185 160
\pinlabel $a$ at 80 160
\pinlabel $d$ at 80 260
\pinlabel $e$ at 335 260
\pinlabel $f$ at 335 160
\endlabellist
\includegraphics[height=7cm]{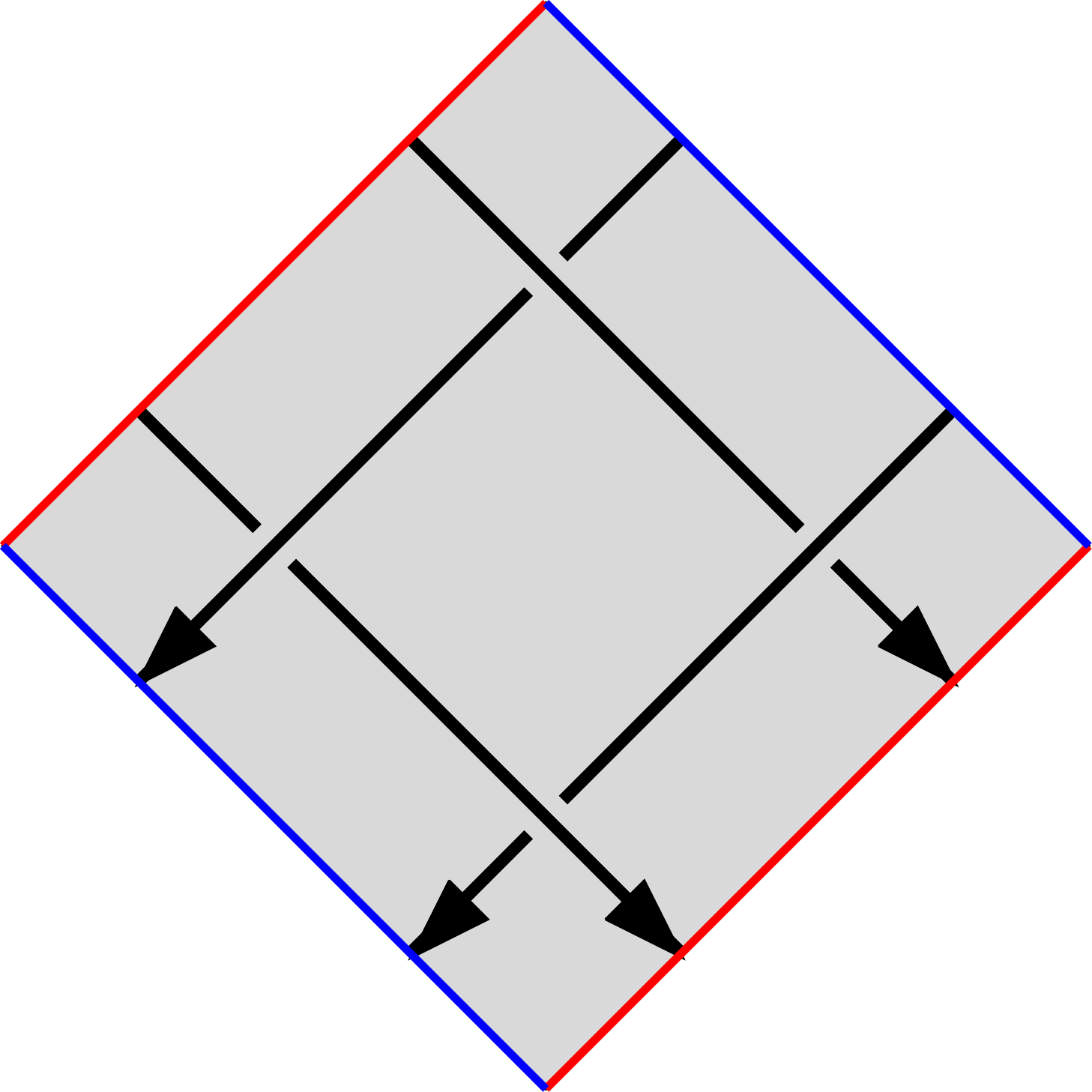}
}
\subcaptionbox{Each crossing is labelled $m$}[.48\textwidth]{
\labellist
\small\hair 2pt
\pinlabel $a$ at 245 365
\pinlabel $b$ at 245 260
\pinlabel $c$ at 245 160
\pinlabel $d$ at 245 95
\pinlabel $c+m$ at 160 95
\pinlabel $b+m$ at 200 365
\pinlabel $a+m$ at 180 240
\pinlabel $d+m$ at 180 180
\pinlabel $a$ at 80 160
\pinlabel $d$ at 80 260
\pinlabel $c+m$ at 325 260
\pinlabel $b+m$ at 325 160
\pinlabel ${\bf m}$ at 215 340
\pinlabel ${\bf m}$ at 215 130
\pinlabel ${\bf m}$ at 340 210
\pinlabel ${\bf m}$ at 80 210
\endlabellist
\includegraphics[height=7cm]{images/square_weave_states_ul}
}
\caption{States of $D$}
\label{fig:sq_weave_states}
\end{figure}

The index $m$ in (\ref{eq:r}) and (\ref{eq:r_inv}) is sometimes thought of as the ``label'' of the associated crossing \cite{gl11}, with the corresponding summand its ``weight." In this way a state of $D$ becomes a labelling of both strands and crossings with integers between $0$ and $n - 1$, and the Kronecker deltas in (\ref{eq:r}) and (\ref{eq:r_inv}) imply we need only consider states whose crossing labels are as shown in Figure \ref{fig:state_rules}.

\begin{figure}[H]
\labellist
\small\hair 2pt
\pinlabel $i$ at 53 110
\pinlabel $j$ at 180 110
\pinlabel $j-m$ at 23 40
\pinlabel $i+m$ at 205 40
\pinlabel $m$ at 85 65
\pinlabel $i$ at 407 110
\pinlabel $j$ at 534 110
\pinlabel $j+m$ at 377 40
\pinlabel $i-m$ at 559 40
\pinlabel $m$ at 439 65
\endlabellist
\centering
\includegraphics[height=1.8cm]{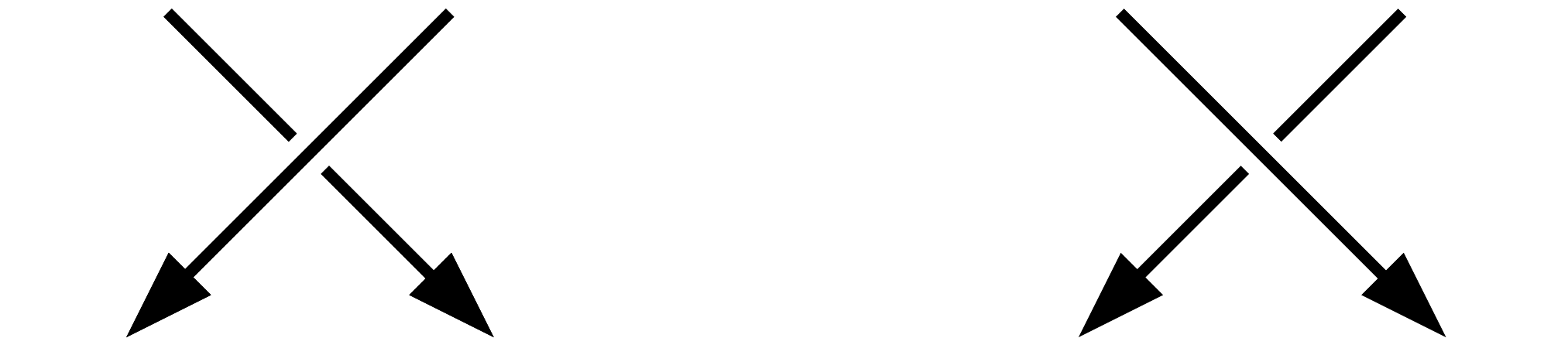}
\caption{States of crossings with non-zero weights}
\label{fig:state_rules}
\end{figure}
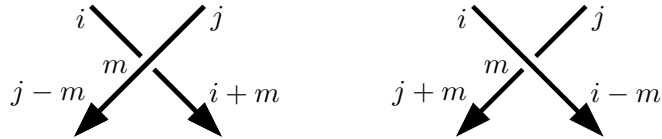

Assigning labels to strands and crossings of $D$ according to these rules gives the diagram in Figure \ref{fig:sq_weave_states}b---the structure of $W$ forces each crossing to have the same label in any nonzero state. We obtain the formula:

\begin{align*}
J^T_n(W;q) &= \sum_{m = 0}^{n - 1}\sum_{a,b,c,d = 0}^{n - m - 1} (R^{-1})^{b + m,a}_{a + m, b}R^{b, c + m}_{c, b + m}R_{a, d + m}^{d, a + m}(R^{-1})^{d+m,c}_{c + m, d} \\
&=  \sum_{m = 0}^{n - 1} \sum_{a,b,c,d = 0}^{n - m - 1} q^{(a - c)(d - b)} \frac{\{a + m\}!\{n - 1 - a\}!}{\{a\}!\{m\}!\{n - 1 - a - m\}!} \cdot \frac{\{b + m\}!\{n - 1 - b\}!}{\{b\}!\{m\}!\{n - 1 - b- m\}!} \\
&\ \ \ \ \ \ \ \ \cdot \frac{\{c + m\}!\{n - 1 - c\}!}{\{c\}!\{m\}!\{n - 1 - c - m\}!} \cdot \frac{\{d + m\}!\{n - 1 - d\}!}{\{d\}!\{m\}!\{n - 1 - d - m\}!}.
\end{align*}

For the remainder of the section, fix $q = e^{2\pi i/n}$. This allows us to apply the identity \cite[eq.~38]{gl11}
$$
\{k\}! = (\sqrt{-1})^{n - 1}\frac{n}{\{n - 1- k\}!}
$$
where $k \in \N$ and $0 \leq k \leq n - 1$, and the formula above becomes
\begin{equation}
\label{eq:formula}
J^T_n(W;q) = \sum_{m = 0}^{n - 1} \sum_{a,b,c,d = 0}^{n - m - 1} q^{(a - c)(d - b)} \frac{1}{(\{m\}!)^4} \cdot \Big( \frac{\{a + m\}!\{b + m\}!\{c + m\}!\{d + m\}!}{\{a\}!\{b\}!\{c\}!\{d\}!}\Big)^2.
\end{equation}

We now prove Theorem \ref{thm:main_one}.

\begin{thm}
\label{thm:main_one}
 $$
\lim_{n \to \infty} \frac{2\pi}{n} \log | J^T_n(W;e^\frac{2\pi i}{n})| = 4 v_\text{oct} = \rm{Vol}((T^2 \times I) \setminus W).
$$
\end{thm}

\begin{proof}
By (\ref{eq:bound}), we need only show
\begin{equation}
\lim_{n \to \infty} \frac{2\pi}{n} \log | J^T_n(W;e^\frac{2\pi i}{n})| \geq 4 v_\text{oct} = \rm{Vol}((T^2 \times I) \setminus W).
\end{equation}

Let $\eta(a,b,c,d) = q^{(a - c)(d - b)}$ and
$$
\rho(a,b,c,d,m) = \frac{1}{(\{m\}!)^4} \cdot \Big( \frac{\{a + m\}!\{b + m\}!\{c + m\}!\{d + m\}!}{\{a\}!\{b\}!\{c\}!\{d\}!}\Big)^2;
$$
then (\ref{eq:formula}) becomes
\begin{equation}
\label{eq:abrev_formula}
J^T_n(W;q) = \sum_{m = 0}^{n - 1} \sum_{a,b,c,d = 0}^{n - m - 1} \eta(a,b,c,d) \cdot \rho(a,b,c,d,m).
\end{equation}
For $k \in \Z$, $0 \leq k \leq n - 1$, $\{k\} = 2i\sin(k\pi /n) = ix$ where $x$ is a non-negative real number. Thus $(\{m\}!)^4$ is a non-negative real number for all values of $m$. Furthermore, since
\begin{align*}
\frac{\{a + m\}!\{b + m\}!\{c + m\}!\{d + m\}!}{\{a\}!\{b\}!\{c\}!\{d\}!} &= \prod_{k = 1}^m \{a + k\}\{b + k\}\{c + k\}\{d + k\}\\
&= 16 \prod_{k = 1}^m \sin \frac{\pi (a + k)}{n}\sin \frac{\pi (b + k)}{n}\sin \frac{\pi (c + k)}{n}\sin \frac{\pi (d + k)}{n},
\end{align*}
$\rho(a,b,c,d,m)$ is a non-negative real number for all relevant values of $a$, $b$, $c$, $d$, and $m$.

If $a = c$, $\eta(a,b,c,d) = 1$ and $\eta(a,b,c,d) \cdot \rho(a,b,c,d,m)$ is a real number. If $a \neq c$, since $\rho(a,b,c,d,m) = \rho(a,c,b,d,m)$ and $\eta(a,b,c,d) = \eta(a,c,b,d)^{-1}$,
\begin{align*}
\eta(a,b,c,d) \rho(a,b,c,d,m) + \eta(a,c&,b,d)\rho(a,c,b,d,m) \\
\ \ \ \ \ &= \big(\eta(a,b,c,d) + \eta(a,b,c,d)^{-1} \big) \cdot \rho(a,b,c,d,m) \\
\ \ \ \ \ &= 2\cos\big(\frac{2\pi(a - c)(d - b)}{n}\big) \cdot \rho(a,b,c,d,m) \in \R.
\end{align*}
Pairing up the summands of (\ref{eq:abrev_formula}) this way we see $J^T_n(W;q)$ is a real number; in fact
\begin{equation}
\label{eq:real_formula}
J^T_n(W;q) = \text{Re}\Big(J^T_n(W;q)\Big) = \sum_{k = 0}^{n - 1} \sum_{a,b,c,d = 0}^{n - k - 1} \cos\big(\frac{2\pi(a - c)(d - b)}{n}\big) \cdot |\rho(a,b,c,d,m)|.
\end{equation}

Using the identity $\cos(\alpha - \beta) = \cos(\alpha)\cos(\beta) + \sin(\alpha)\sin(\beta)$, we rewrite (\ref{eq:real_formula}) as
\begin{align}
J^T_n(W;q) &= \sum_{m = 0}^{n - 1} \sum_{a,b,c,d = 0}^{n - m - 1} \cos\big(\frac{2\pi(a - c)(d - b)}{n}\big) \bigg| \frac{\{a + m\}!\{b + m\}!\{c + m\}!\{d + m\}!}{\{a\}!\{b\}!\{c\}!\{d\}!}\bigg|^2 \bigg|\frac{1}{\{m\}!} \bigg|^4 \nonumber \\
&= \sum_{m = 0}^{n - 1} \sum_{b,d = 0}^{n - m - 1} \bigg|\frac{1}{\{m\}!} \bigg|^4  \bigg| \frac{\{b + m\}!\{d + m\}!}{\{b\}!\{d\}!}\bigg|^2 \cdot \nu(b,d,m) \label{eq:nu_formula}
\end{align}
Where
\begin{align*}
\nu(b,d,m) &= \sum_{a,c = 0}^{n - m - 1} \cos\big(\frac{2\pi a(d - b)}{n} - \frac{2\pi c(d - b)}{n}\big) \bigg| \frac{\{a + m\}!\{c + m\}!}{\{a\}!\{m\}!}\bigg|^2 \\
&= \sum_{a,c = 0}^{n - m - 1} \cos\big(\frac{2\pi a(d - b)}{n}\big)\cos \big( \frac{2\pi c(d - b)}{n}\big) \bigg| \frac{\{a + m\}!\{c + m\}!}{\{a\}!\{m\}!}\bigg|^2 \\
&\ \ \ \ \ \ \ \ + \sum_{a,c = 0}^{n - m - 1} \sin\big(\frac{2\pi a(d - b)}{n}\big)\sin \big( \frac{2\pi c(d - b)}{n}\big) \bigg| \frac{\{a + m\}!\{c + m\}!}{\{a\}!\{m\}!}\bigg|^2 \\
&= \bigg( \sum_{a = 0}^{n - m - 1}\cos\big(\frac{2\pi a(d - b)}{n}\big) \bigg| \frac{\{a + m\}!}{\{a\}!}\bigg|^2\bigg)^2 + \bigg( \sum_{a = 0}^{n - m - 1}\sin\big(\frac{2\pi a(d - b)}{n}\big) \bigg| \frac{\{a + m\}!}{\{a\}!}\bigg|^2\bigg)^2.
\end{align*}
In particular, $\nu(b,d,m)$ is a non-negative real number for all values of $b$,$d$, and $m$.

Because each summand of (\ref{eq:nu_formula}) is non-negative and real, $J^T_n(W; q)$ is bounded below for all $n$ by the summand of (\ref{eq:nu_formula}) with $m = \lfloor n/2 \rfloor$ and $b = d = \lfloor n/4 \rfloor$. Additionally, $\nu(\lfloor n/4 \rfloor,\lfloor n/4 \rfloor,\lfloor n/2 \rfloor)$ is bounded below by the summand of the equation above with $a = \lfloor n/4 \rfloor$.  We have
\begin{align}
J^T_n(W;q) &\geq \bigg|\frac{1}{\{\lfloor n/2 \rfloor\}!} \bigg|^4  \bigg| \frac{\{\lfloor n/4 \rfloor + \lfloor n/2 \rfloor\}!\{\lfloor n/4 \rfloor + \lfloor n/2 \rfloor\}!}{\{\lfloor n/4 \rfloor\}!\{\lfloor n/4 \rfloor\}!}\bigg|^2 \nu(\lfloor n/4 \rfloor,\lfloor n/4 \rfloor,\lfloor n/2 \rfloor) \nonumber \\
&\geq \bigg|\frac{1}{\{\lfloor n/2 \rfloor \}!} \bigg|^4 \bigg| \frac{\{\lfloor n/4 \rfloor + \lfloor n/2 \rfloor \}!}{\{\lfloor n/4 \rfloor \}!}\bigg|^8 \nonumber \\
&\geq \text{min} \bigg(\bigg|\frac{1}{\{\lfloor n/2 \rfloor \}!} \bigg|^4 \bigg| \frac{\{\lfloor 3n/4 \rfloor \}!}{\{\lfloor n/4 \rfloor \}!}\bigg|^8, \bigg|\frac{1}{\{\lfloor n/2 \rfloor \}!} \bigg|^4 \bigg| \frac{\{\lfloor 3n/4 - 1 \rfloor \}!}{\{\lfloor n/4 \rfloor \}!}\bigg|^8 \bigg). \label{eq:second_last}
\end{align}
Garoufalidis and L{\^e} \cite{gl11} proved that, for $\alpha \in (0,n)$,
$$
\log|\{\lfloor \alpha \rfloor\}!| = -\frac{n}{\pi} \Lambda(\pi \frac{\alpha}{n}) + O(\log n).
$$
Here $\Lambda$ is the Lobachevsky function $\Lambda(z) = -\int_0^z \log | 2 \sin \zeta | d \zeta$ and $O(\log n)$ is an expression bounded by $C\log n$ for a constant $C$ independent of $n$.
Applying this to (\ref{eq:second_last}) gives
\begin{align*}
\lim_{n \to \infty}  \frac{2\pi}{n} \log | J^T_n(W;q) | &\geq \lim_{n \to \infty} 8 \Lambda(\frac{\pi}{2}) + 16\Lambda(\frac{\pi}{4}) - 16 \cdot \text{min}\big(\Lambda(\frac{3\pi}{4}), \Lambda(\frac{(3n - 4)\pi}{4n}) \big) + \frac{O\log(n)}{n} \\
&= 4(2\Lambda(\frac{\pi}{2}) + 4\Lambda(\frac{\pi}{4}) - 4\Lambda(\frac{3\pi}{4})) = 4v_\text{oct},
\end{align*}
proving the theorem.
\end{proof}

\section{Generalizing the Volume Conjecture For Links in $T^2 \times I$}

\subsection{Simplicial Volume}

With the original Volume Conjecture \ref{thm:vol_conj} in mind, we generalize our Volume Conjecture \ref{thm:conj} to links which may not be hyperbolic.
\begin{conj}
\label{thm:broad_conj}
For any link $L \subset T^2 \times I$ such that $(T^2 \times I) \setminus L$ is irreducible,
\begin{equation}
\label{eq:broad_conj}
\lim_{n \to \infty} \frac{2\pi}{n} \log | J^T_n(L;e^\frac{2\pi i}{n})| = \rm{Vol}((T^2 \times I) \setminus L)
\end{equation}
where $n > 0$ runs over all odd integers.
\end{conj}

As in Conjecture \ref{thm:vol_conj}, Vol refers to \emph{simplicial volume}---the sum of the volumes of the hyperbolic pieces in the JSJ decomposition of $(T^2 \times I) \setminus L$. By \emph{irreducible}, we mean that every smooth embedded $2$-sphere in $(T^2 \times I) \setminus L$ bounds a $3$-ball. The irreducibility condition and the restriction to odd $n$ are, in fact, necessary if one wishes to generalize the original Volume Conjecture \ref{thm:vol_conj} from knots to links. For a link $L \subset S^3$, $S^3 \setminus L$ being irreducible is equivalent to $L$ not being a split link, a class of links for which the colored Jones polynomial is known to vanish \cite{mm01}. Separately, Van der Veen has constructed a class of non-split links called \emph{Whitehead chains} for which the colored Jones polynomial vanishes at even values of $n$ \cite{v08}. Before discussing the necessity of these two conditions in Conjecture \ref{thm:broad_conj}, we give some positive results.

Call $L' \subset S^3$ a \emph{VC-verified link} if the Volume Conjecture \ref{thm:vol_conj} is known to hold for $L'$. That is, $L'$ is VC-verified if
$$
\lim_{n \to \infty} \frac{2\pi}{n} \log | J_n(L';e^\frac{2\pi i}{n})| = \rm{Vol}(S^3 \setminus L'),
$$
where the limit runs over odd $n > 0$. VC-verified links include the figure eight knot, the Borromean rings, and others---see \cite[Ch.~3]{my18} for a somewhat recent, comprehensive list.

\begin{thm}
\label{thm:hitoshi}
Let $L' \subset S^3$ be a VC-verified link, and consider an inclusion of $L'$ in an embedded $2$-sphere in $T^2 \times I$. Let $K \subset T^2 \times I$ be a knot projecting to an essential, simple closed curve in $T^2 \times \{0\}$, and let $L$ be a connect sum $L = L' \# K$. Then Conjecture \ref{thm:broad_conj} holds for $L$.
\end{thm}

See Figure \ref{fig:t_to_t} for an example where $L'$ is the figure eight knot and $K$ is a meridian. The main ingredient in the proof of Theorem \ref{thm:hitoshi} is the following relationship between $J^T_n$ and $J_n$.

\begin{thm}
\label{thm:relationship}
Let $L'$ be a link in $S^3$, and consider an inclusion of $L'$ in an embedded $2$-sphere in $T^2 \times I$. Let $K \subset T^2 \times I$ be a knot projecting to an essential, simple closed curve in $T^2 \times \{0\}$, and let $L$ be a connect sum $L = L' \# K$. Then
$$
J^T_n(L;q) = n \cdot J_n(L';q).
$$
\end{thm}
\begin{proof}
Using Proposition \ref{thm:follow_up}, we assume $K$ is a meridian. Then we can choose a diagram $D \subset T^2$ for $L$ and a lift, $\tilde{D}$, of $D$ to $\R^2$ such that $D' = \tilde{D} \cap I^2$ is a diagram of $L'$ as a $(1,1)$-tangle. (See Figure \ref{fig:t_to_t}, where $L'$ is the figure eight knot.) Coloring $D'$ by $V^n$, Theorem \ref{thm:fund_tangles} associates an $\mathscr{A}$-linear map $\phi : V^n \to V^n$ to $D'$. The irreducibility of $V^n$ implies $\phi$ is a scalar multiple of the identity, and, after accounting for writhe, this scalar is $J_n(L';q)$---see \cite[Lem. 3.9]{km91} and \cite{gl11, my18}.

On the other hand, by the proof of Lemma \ref{thm:basis_invar}, $J^T_n(L, q)$ is the trace of $\phi$ (corrected for writhe). We conclude 
$$
J^T_n(L;q) = n \cdot J_n(L';q).
$$
\end{proof}

\begin{figure}[H]
\centering
\includegraphics[height=4cm]{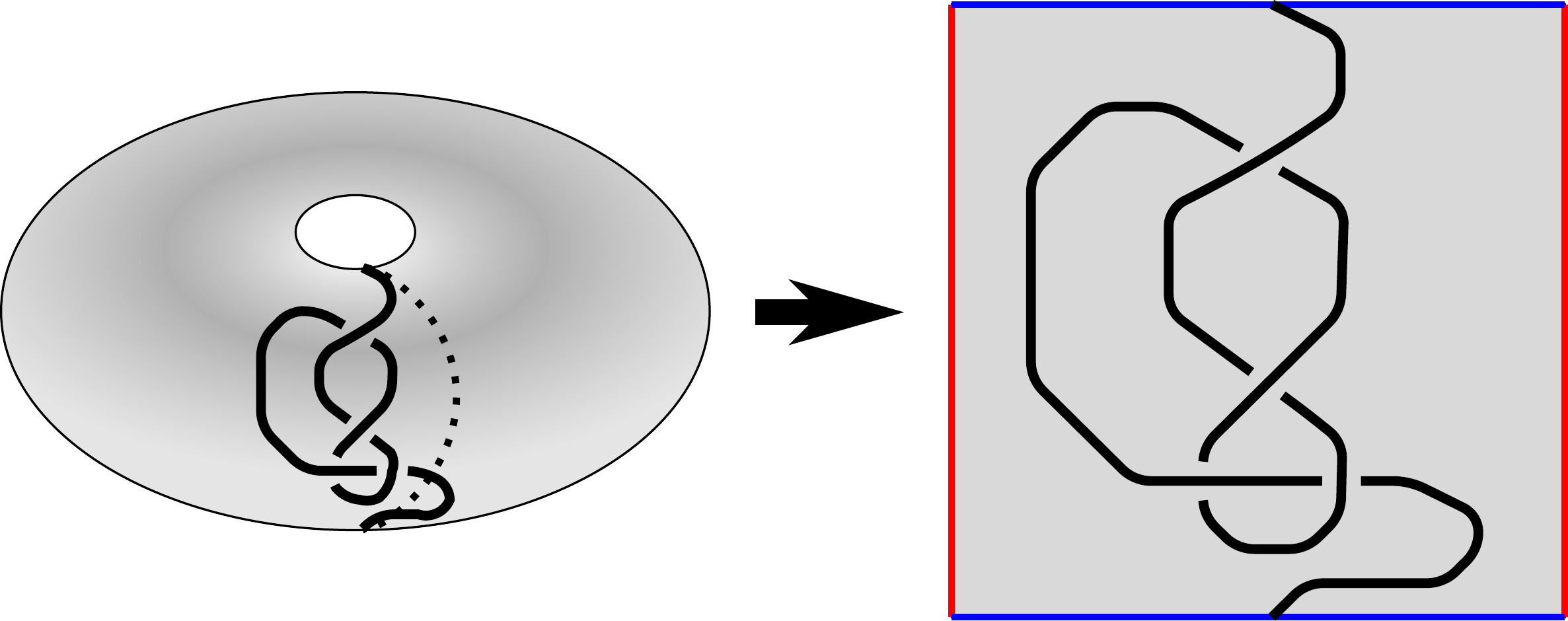}
\caption{Composing the figure eight knot with a meridian in $T^2 \times I$}
\label{fig:t_to_t}
\end{figure}

Theorem \ref{thm:hitoshi} follows.

\begin{proof}[Proof of Theorem \ref{thm:hitoshi}]
By Theorem \ref{thm:relationship}, $J^T_n(L;q) = n \cdot J_n(L';q)$. Therefore
\begin{equation}
\label{eq:jones_equality}
\lim_{n \to \infty} \frac{2\pi}{n} \log | J^T_n(L;e^\frac{2\pi i}{n})| = \lim_{n \to \infty} \frac{2\pi}{n} \log | J_n(L';e^\frac{2\pi i}{n})|.
\end{equation}
The complement $(T^2 \times I) \setminus K$ is homeomorphic to $S^3 \setminus H'$, where $H'$ is the link shown in Figure \ref{fig:tangle_hopf}a. This implies $(T^2 \times I) \setminus L$ is homeomorphic to $S^3 \setminus (L' \# H')$, where the composition is formed as in Figure \ref{fig:tangle_hopf}b. By \cite{s81},
\begin{equation}
\label{eq:simp_vols}
\rm{Vol}((T^2 \times I) \setminus L) = \rm{Vol}(S^3 \setminus (L' \# H')) =  \rm{Vol}(S^3 \setminus L') + \rm{Vol}(S^3 \setminus H') = \rm{Vol}(S^3 \setminus L'),
\end{equation}
and combining (\ref{eq:jones_equality}) and (\ref{eq:simp_vols}) gives
$$
\lim_{n \to \infty} \frac{2\pi}{n} \log | J^T_n(L;e^\frac{2\pi i}{n})| = \lim_{n \to \infty} \frac{2\pi}{n} \log | J_n(L';e^\frac{2\pi i}{n})| = \rm{Vol}(S^3 \setminus L') = \rm{Vol}((T^2 \times I) \setminus L).
$$
\end{proof}

\begin{figure}[H]
\centering
\subcaptionbox{}[.48\textwidth]{
\includegraphics[height=3.5cm]{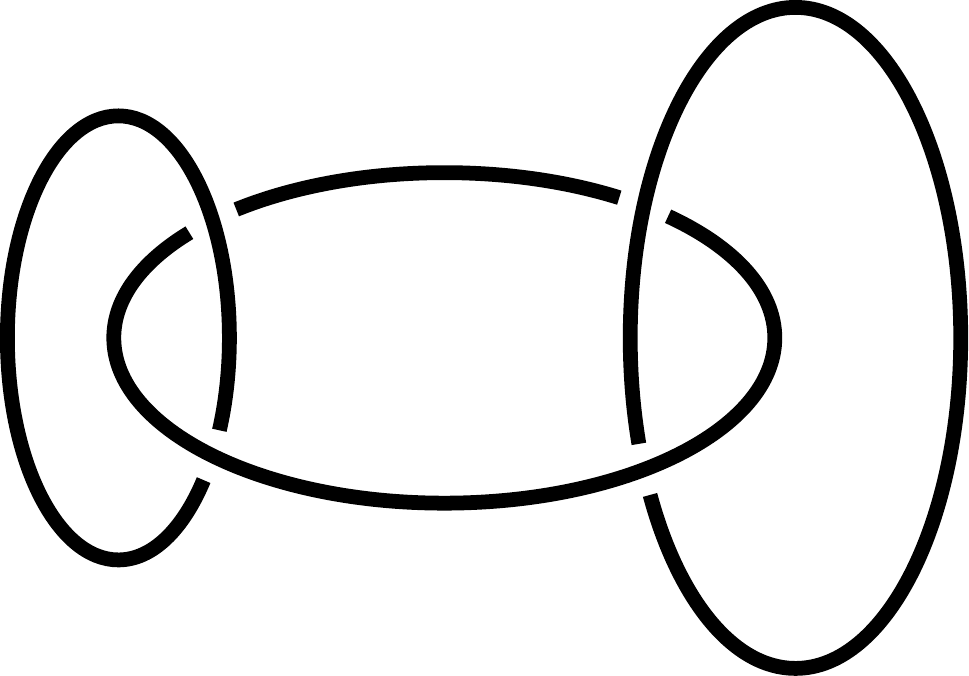}
}
\subcaptionbox{}[.48\textwidth]{
\includegraphics[height=3.5cm]{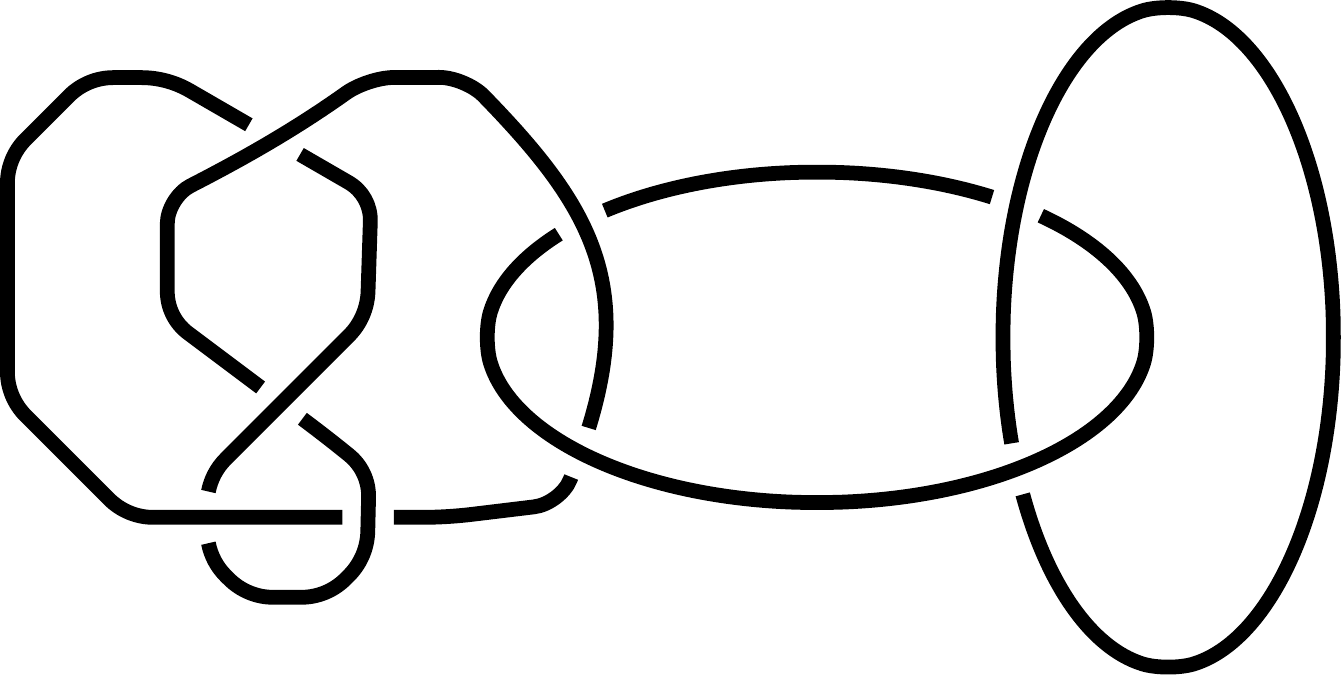}
}
\caption{The links $H'$ and $L' \# H'$ in $S^3$, from the proof of Theorem \ref{thm:hitoshi}, with $L'$ the figure eight knot}
\label{fig:tangle_hopf}
\end{figure}

We've shown any positive result for the original Volume Conjecture \ref{thm:vol_conj} gives a positive result for Conjecture \ref{thm:broad_conj}. The proof also shows why restricting to odd $n$ is necessary---if we let $L' \subset S^3$ be a Whitehead chain, as in \cite{v08}, the link $L' \# K \subset T^2 \times I$ (defined as above) will satisfy Conjecture \ref{thm:broad_conj} but the toroidal colored Jones polynomial will vanish for even $n$. 

Using the nice behavior of simplicial volume and the toroidal colored Jones polynomial under split unions of links, we can push the result of Theorem \ref{thm:hitoshi} further. We define a \emph{split union} $L = L_1 \sqcup L_2$ of links $L_1, L_2 \subset T^2 \times I$ to be a union such that $L$ admits a torus diagram which is a disjoint union of diagrams of $L_1$ and $L_2$. Additionally, define a \emph{torus link} to be a link in $T^2 \times I$ with a diagram consisting of a set of disjoint, simple closed curves in $T^2$.

\begin{cor}
Let $L_1', L_2', \dots, L_m' \subset S^3$ be VC-verified links, and define $K$ as in Theorem \ref{thm:hitoshi}. Let $L_i = L_i' \# K$ for $i = 1, \dots, m$. Then Conjecture \ref{thm:broad_conj} holds for the split union $L = L_1 \sqcup L_2 \sqcup \dots \sqcup L_m$. In particular, Conjecture \ref{thm:broad_conj} holds for all torus links with no nullhomotopic components.
\end{cor}

\begin{proof}
The result follows just as in Theorem \ref{thm:hitoshi} after checking that
$$
J^T_n(L) = J^T_n(L_1) \cdot J^T_n(L_2) \cdots J^T_n(L_m)
$$
and
$$
\rm{Vol}(L) = \rm{Vol}(L_1) + \rm{Vol}(L_2) + \cdots + \rm{Vol}(L_m).
$$
To prove the second statement, let $L_i'$ be the unknot for all $i$---every torus link with no nullhomotopic components can be obtained this way. Alternatively, one could use Proposition \ref{thm:why_two} and a direct computation. The result also holds for torus links with nullhomotopic components (see Proposition \ref{thm:nullh} below), but the complement of such a link in $T^2 \times I$ is not irreducible.
\end{proof}

If we view links in the thickened torus as generalizations of $(1,1)$-tangles, as the proof of Theorem \ref{thm:relationship} suggests, and think of the colored Jones polynomial as an invariant of $(1,1)$-tangles, the toroidal colored Jones polynomial becomes a generalization of the colored Jones polynomial rather than a toroidal analogue. This view is supported by Corollary \ref{thm:cor_imp} below, which shows Conjecture \ref{thm:broad_conj} implies the original Volume Conjecture \ref{thm:vol_conj}.

\begin{cor}
\label{thm:cor_imp}
For knots, Conjecture \ref{thm:broad_conj} implies the original Volume Conjecture \ref{thm:vol_conj}.
\end{cor}
\begin{proof}
Given a knot $K' \subset S^3$, let $L = K' \# K$ with $K$ defined as above. Then, assuming Conjecture \ref{thm:broad_conj},
$$
\lim_{n \to \infty} \frac{2\pi}{n} \log | J_n(K';e^\frac{2\pi i}{n})| = \lim_{n \to \infty} \frac{2\pi}{n} \log | J^T_n(L;e^\frac{2\pi i}{n})| = \rm{Vol}((T^2 \times I) \setminus L) = \rm{Vol}(S^3 \setminus K').
$$
\end{proof}

In this sense, Conjecture \ref{thm:broad_conj} generalizes Conjecture \ref{thm:vol_conj}. It is interesting to note that Conjecture \ref{thm:vol_conj} does not seem to imply Conjecture \ref{thm:broad_conj}.

As we noted earlier, just as the original Volume Conjecture \ref{thm:vol_conj} fails for split links \cite{mm01}, Conjecture \ref{thm:broad_conj} fails for links in $T^2 \times I$ which have one or more nullhomotopic split components. By a \emph{nullhomotopic split component} of a link $L \subset T^2 \times I$, we mean a sublink $L' \subset L$ such that $L'$ is contained in an embedded $2$-sphere in $(T^2 \times I) \setminus L$, and no proper sublink of $L'$ is contained in such a sphere. This is implied by the following:

\begin{prop}
\label{thm:nullh}
If $L \subset T^2 \times I$ is a link with a nullhomotopic split component, $J^T_n(L; e^{2\pi i/n}) = 0$ for all $n$. In particular, if $L$ is nullhomotopic, $J^T_n(L; e^{2\pi i/n}) = 0$ for all $n$.
\end{prop}

\begin{proof}
Let $L \subset T^2 \times I$ be a link with nullhomotopic split component $L_1$, and let $L_2 = L \setminus L_1$. Then $J^T_n(L;q) = J^T_n(L_1;q) \cdot J^T_n(L_2;q)$, so it suffices to show $J_n^T(L_1, e^{2\pi i/n}) = 0$.

Since $L_1$ is nullhomotopic, it has a torus diagram $D$ which lifts to a diagram $\tilde{D} \subset \R^2$ such that $\tilde{D} \cap I^2$ is a diagram for $L_1$ as a link in $S^3$. See Figure \ref{fig:figure_eight}, where $L_1$ is the figure eight knot. A direct computation shows $J^T_n(L_1;q) = [n] J_n(L_1;q)$, and $[n] = 0$ when $q = e^{2\pi i/n}$.
\end{proof}

\begin{figure}[H]
\centering
\includegraphics[height=4cm]{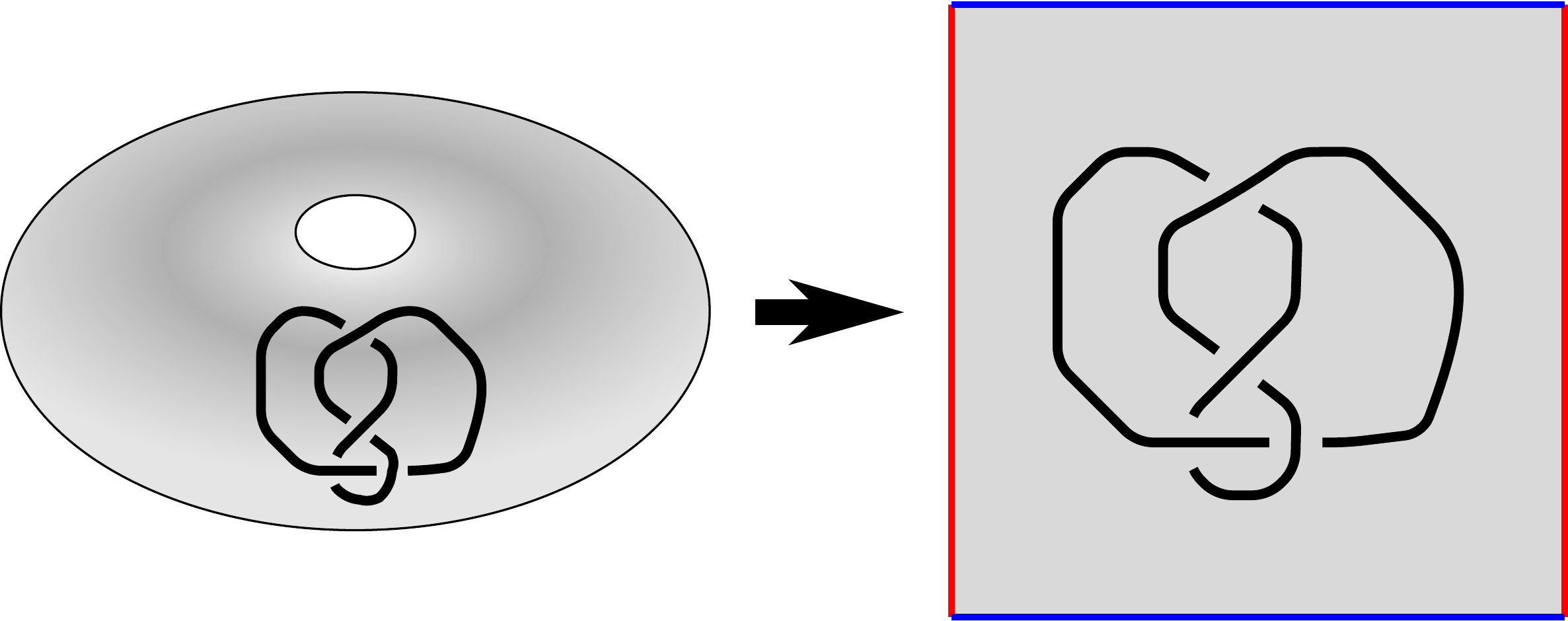}
\caption{A nullhomotopic inclusion of the figure eight knot in $T^2 \times I$}
\label{fig:figure_eight}
\end{figure}

\begin{rmk} In \cite{v09}, Van der Veen noted that the original Volume Conjecture \ref{thm:vol_conj} can be changed to account for split links by choosing a different normalization of the colored Jones polynomial. Essentially, each split component adds a factor of $[n]$ to $J_n$---if a link $L$ has $s$ split components, we can divide by $[n]^s$ to obtain a non-zero value at the root $q = e^{2\pi i/n}$.

Analogously, if $L \subset T^2 \times I$ has $s$ \textit{nullhomotopic} split components, we can ask whether
$$
\lim_{n \to \infty} \frac{2\pi}{n} \log | \frac{1}{[n]^s} J^T_n(L;e^\frac{2\pi i}{n})| = \rm{Vol}((T^2 \times I) \setminus L).
$$
Replacing equation (\ref{eq:broad_conj}) in Conjecture \ref{thm:broad_conj} with the above equation, we can remove the hypothesis that $(T^2 \times I) \setminus L$ be irreducible.
\end{rmk}

\subsection{Higher-Genus Surfaces}

Taking a different direction, one could attempt to generalize Conjecture \ref{thm:conj} to links in thickened surfaces of genus greater than one. As we noted earlier, while there is no obvious way to define pseudo-operator invariants for links in these surfaces, Corollary \ref{thm:skein_cjp} lets us define the $SU(2)$ toroidal colored Jones polynomial skein-theoretically in any orientable manifold.

As defined, volume conjecture behavior is unlikely to occur in thickened surfaces of genus greater than one. To see why, let $\Sigma \times I$ be such a thickened surface containing a link $L$. Since $\Sigma \times I$ has boundary components which are not spheres or tori, there is not a unique way to assign a complete hyperbolic structure to the complement of $L$. One way to resolve this ambiguity, as in \cite{a19}, is to choose the hyperbolic structure on $(\Sigma \times I) \setminus L$ which has totally geodesic boundary. If such a structure exists, $(\Sigma \times I) \setminus L$ is called \emph{tg-hyperbolic} and it has a finite \emph{tg-hyperbolic volume}.

Theorem \ref{thm:gl} says that, in the case of a link $L$ in the thickened torus with crossing number $c$,
\begin{equation}
\label{eq:bound2}
\lim_{n \to \infty} \frac{2\pi}{n} \log | J^T_n(L;e^\frac{2\pi i}{n})| \leq c \cdot v_\text{oct}.
\end{equation}
A similar bound exists for links in $S^3$---see \cite[Thm.~1.13]{gl11}---and we conjecture that (\ref{eq:bound2}) holds for $J^T_n$ for links in any genus thickened surface. In surfaces with genus greater than one, however, there are many links whose tg-hyperbolic volume exceeds this bound. Consider, for example, the virtual link 3.1 of \cite{g04} viewed as a link in the thickened orientable surface of genus two: its crossing number is three and its tg-hyperbolic volume is $\approx 18.75 > 3v_\text{oct}$ \cite{a19}. Thus, volume convergence as defined above is not possible if (\ref{eq:bound2}) holds for genus two surfaces and we choose the tg-hyperbolic structure on the complement of $L$.

This does not mean \emph{no} volume conjecture can exist for links in higher genus surfaces---just that any such conjecture would need to look different from Conjecture \ref{thm:vol_conj} and Conjecture \ref{thm:conj}. It may be interesting to examine what kind of relationship can exist between the $SU(2)$ toroidal colored Jones polynomial of a link in a higher-genus surface and its tg-hyperbolic volume.

\section{The Toroidal Colored Jones Polynomial as an Invariant of Biperiodic and Virtual Links}

Beyond its volume conjecture behavior, the toroidal colored Jones polynomial may be useful as an invariant of biperiodic and virtual links. A \emph{biperiodic link} is a properly embedded $1$-manifold $\tilde{L} \subset \R^2 \times I$, such that $\tilde{L}$ is invariant under translations by a $2$-dimensional lattice $\Lambda$ and $L = \tilde{L} / \Lambda$ is a link in $T^2 \times I$---see \cite{ckp19}. We call $\Lambda$ \emph{maximal} if it is not properly contained in another invariant lattice for $\tilde{L}$, in which case the resulting link $L \subset T^2$ is a \emph{minimal representative} of $\tilde{L}$. For a given biperiodic link $\tilde{L}$, there are many possible choices of minimal representative. However, if $L_1, L_2 \subset T^2 \times I$ are two minimal representatives of $\tilde{L} \subset \R^2 \times I$ with respective diagrams $D_1, D_2 \subset T^2$, then there exists an orientation-preserving homeomorphism $f$ of $T^2$ such that $f(D_1) = D_2$ (c.f. \cite[Prop.~2.1]{gmo07}). Hence, Proposition \ref{thm:follow_up} gives the following:

\begin{thm}
\label{thm:biperiodic_links}
If $\tilde{L} \subset \R^2 \times I$ is a biperiodic link and $L \subset T^2 \times I$ is a minimal representative of $\tilde{L}$, define $J_n^T(\tilde{L}) = J_n^T(L)$. Then $J^T_n$ is an invariant of biperiodic links in $\R^2 \times I$.
\end{thm}

Another non-classical type of link, \emph{virtual links}, are an area of extensive study---see \cite{k99} for an introduction. By \cite{cks02, k03}, any virtual link $L'$ is represented uniquely by a link $L$ in a minimal-genus thickened surface, up to an orientation-preserving homeomorphism of the surface. The $U_q(sl(2, \C))$ toroidal colored Jones polynomial is defined only for links in $T^2 \times I$, but the $SU(2)$ toroidal colored Jones polynomial can be defined skein-theoretically for links in any thickened surface. Similar to above, we have:

\begin{thm}
\label{thm:virtual_links}
If $L'$ is a virtual link and $L \subset \Sigma \times I$ is a minimal representative of $L'$, define $J_n^T(L') = J_n^T(L)$. Then $J^T_n$ is an invariant of virtual links.
\end{thm}
Here $\Sigma$ is a closed, orientable surface and $J^T_n$ is the $SU(2)$ toroidal colored Jones polynomial, defined skein-theoretically as in Corollary \ref{thm:skein_cjp}. To prove Theorem \ref{thm:virtual_links}, we need only show the skein-theoretic $J^T_n$ is preserved by orientation-preserving homeomorphisms of surfaces. This is done in the lemma below.

\begin{lemma}
\label{thm:follow_up_again}
Let $L, L'$ be links in $\Sigma \times I$ with respective diagrams $D, D' \subset \Sigma$, $\Sigma$ a closed, orientable surface.  If $f$ is an orientation-preserving homeomorphism of $\Sigma$ satisfying $f(D) = D'$, then $J^T_n(L;q) = J^T_n(L';q)$ for all $n \in \N$. Here $J^T_n$ is the $SU(2)$ toroidal Jones polynomial, defined skein-theoretically.
\end{lemma}

\begin{proof}
Because $f$ preserves orientation, $D$ and $D'$ have the same writhe. Thus it suffices to prove the result for $\hat{J}^T_{{\bf n}, q}$, which follows  from the case of $\hat{J}^T_{{\bf 2}, q}$. Equivalently, we show the bracket $\langle * \rangle_\tau$ defined in Section 5 is invariant under orientation-preserving homeomorphisms of $\Sigma$.

The claim follows by induction on crossing number, noting $f$ induces a bijection on the crossings of $D$ and $D'$. If $D$ has no crossings, $\langle D \rangle_{\tau}$ is determined by whether or not $D$ is contractible, which is preserved by $f$. For an arbitrary diagram $D$, we can ``resolve'' a crossing using the relation (c) of Definition \ref{def:kauffman}. Since $f$ commutes with both types of crossing resolution in relation (c), the claim follows inductively.
\end{proof}

As Remark \ref{rmk:virtual} discusses, $J^T_n$ is distinct from existing quantum invariants of virtual links. To our knowledge, it is the first invariant of virtual links to exhibit volume conjecture behavior for genus one virtual links, i.e. links in the thickened torus. Continuing our discussion from Section 7.2, it is interesting to ask what kind of volume conjecture behavior emerges in higher-genus virtual links.

\appendix

\section{The Toroidal Colored Jones Polynomial and Rotation Number}

The following generalization of property (b) of Lemma \ref{thm:pre_kauffman} is not hard to prove, using Proposition \ref{thm:why_two} and a direct computation.
\begin{prop}
\label{thm:torus_knots}
Let $K \subset T^2 \times I$ be a knot projecting to a simple closed curve in $T^2$.
\begin{enumerate}[label=(\alph*)]
\item If $K$ is nullhomotopic, the $SU(2)$ toroidal colored Jones polynomial $J_n^T$ satisfies
$$
J_n^T(K;q) = -[n].
$$
The $U_q(sl(2,\C))$ toroidal colored Jones polynomial $J_n^T$ satisfies
$$
J_n^T(K;q) = [n].
$$
\item If $K$ is not nullhomotopic, the $SU(2)$ and $U_q(sl(2,\C))$ toroidal colored Jones polynomials both satisfy
$$
J_n^T(K;q) = n.
$$
\end{enumerate}
\end{prop}

Proposition \ref{thm:torus_knots} says, in a sense, that contractible, simple closed curves in $T^2$ are ``quantized'' by the toroidal colored Jones polynomial while essential, simple closed curves are not. We would like to motivate geometrically why this striking phenomenon occurs.

To accomplish this, we recall Lin and Wang's definition of the Jones polynomial \cite{lw01}, adapted from work in \cite{t88-2}. As we will see, their construction extends in a natural way to define $J^T_2$ and, by cabling, $J^T_n$ for all $n > 2$. Its use of rotation number provides insight into Proposition \ref{thm:torus_knots}, at least for $n = 2$.

We briefly recall Lin and Wang's definition. First, fix the preferred basis $\{e_0, e_1\}$ of $V^2$ we used in Section $7$. In this basis the $R$-matrix coefficients are:
\begin{align*}
R^{0,0}_{0,0} &= R^{1,1}_{1,1} = q^{1/4}, &R^{1,0}_{0,1} &= R^{0,1}_{1,0} = q^{-1/4}, &R^{0,1}_{0,1} = q^{1/4} - q^{-3/4} \\
(R^{-1})^{0,0}_{0,0} &= (R^{-1})^{1,1}_{1,1} = q^{-1/4}, &(R^{-1})^{1,0}_{0,1} &= (R^{-1})^{0,1}_{1,0} = q^{1/4}, &(R^{-1})^{1,0}_{1,0} = q^{-1/4} - q^{3/4}
\end{align*}
and all other entries of $R$ and $R^{-1}$ are zero.

Given a diagram $D$ of an oriented link $L \subset S^3$, let $P_c$ be the set a crossing points of $D$. In this context, a state $s$ of $D$ is an assignment of $0$ or $1$ to each component of $D \setminus P_c$. (States are defined differently here than in Section 3---we ignore local extrema and do not make use of $(V^2)^*$.) If a state $s$ labels a neighborhood of a positive crossing $p$ with $i,j,k,l \in \{0,1\}$ as in Figure \ref{fig:crit_point}, the weight of the crossing is $\omega_p(s) = R^{ij}_{kl}$. If $s$ labels a neighborhood of a negative crossing the same way, the weight of $p$ is $\omega_p(s) = (R^{-1})^{ij}_{kl}$.

Similar to (\ref{eq:state_sum_1}), we define the total weight of a state $s$ to be
$$
\omega^c(s) = \prod_{p \in P_c} \omega_p(s).
$$
A state $s$ is called \emph{admissible} if $\omega^c(s) \neq 0$. Examining the coefficients of $R$ and $R^{-1}$, we see $s$ is admissible if and only if each crossing of $D$ has one of the patterns of labels shown in Figure \ref{fig:adm_crossing}, where dashed and solid lines indicate $0$- and $1$-labels respectively. If either of the two rightmost cases in Figure \ref{fig:adm_crossing} occurs in $D$, we resolve the given crossing into two vertical lines. This decomposes $D$ into a set of closed curves, each labelled entirely by $0$ or entirely by $1$ in $s$. Define $\text{rot}_i(D, s)$ to be the sum of the rotation numbers (the degree of the Gauss map) of all $i$-labelled curves of $D$ after these resolutions take place. Then
\begin{prop}[\cite{lw01}]
\label{thm:lin_wang}
$$
J_2(L; q) = \frac{1}{[n]} (q^{3/4})^{-w(D)} \sum_{s \in \text{Adm}_c(D)} q^{(\text{rot}_1(D,s) - \text{rot}_0(D,s))/2} \cdot \omega^c(s),
$$
where $\text{Adm}_c(D)$ is the set of admissible states of $D$.
\end{prop}

\begin{figure}[H]
\centering
\includegraphics[height=1.5cm]{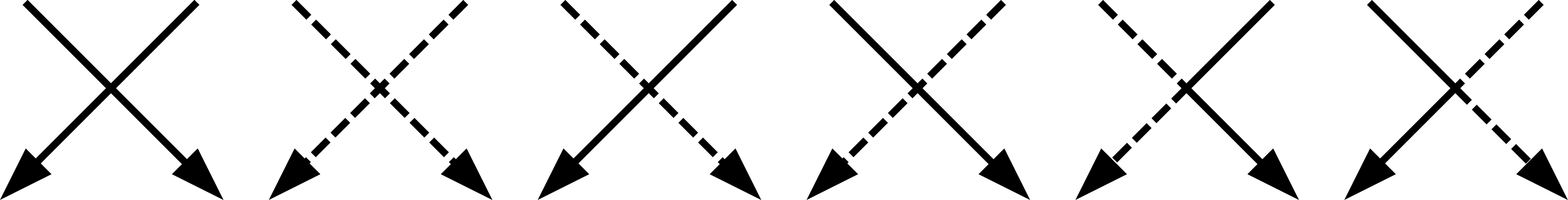}
\caption{Admissible states near crossings}
\label{fig:adm_crossing}
\end{figure}

Removing the factor of $1/[n]$, this definition extends to a torus with no trouble. It agrees with our definition of $J^T_2$.
\begin{thm}
Let $L \subset T^2 \times I$ be an oriented link with diagram $D \subset T^2$. Then
\label{thm:me_and_lw}
\begin{equation}
\label{eq:lw}
J^T_2(L; q) = (q^{3/4})^{-w(D)} \sum_{s \in \text{Adm}_c(D)} q^{(\text{rot}_1(D,s) - \text{rot}_0(D,s))/2} \cdot \omega^c(s),
\end{equation}
where all terms are defined as in Proposition \ref{thm:lin_wang}.
\end{thm}

\begin{proof}
We sketch the proof. Let $P$ denote the set of crossing points and local extrema of $D$, as in Section $3$, and use $\sigma$ to denote a state of $D$ in the pseudo-operator invariant context (see (\ref{eq:state_sum_1}) and the preceding discussion). Call a state $\sigma$ \emph{admissible} if $\omega(\sigma) \neq 0$, and let $\text{Adm}(D)$ be the set of admissible states of $D$ in this context.

In the given basis for $V^2$, the operator $\mu : V^2 \to V^2$ is defined by
\begin{equation}
\label{eq:mu}
\mu_0^0 = q^{-1/2}, \ \ \ \ \ \mu_1^1 = q^{1/2},
\end{equation}
and all other coefficients are zero \cite{km91}. Thus, a state $\sigma$ is admissible only if both sides of every extreme point of $D$ are assigned the same number, either $0$ or $1$. (Here $i \in \{0,1\}$ might refer to the basis element $e_i$ or the dual element $e^i$.) It follows that $\text{Adm}(D)$ is in bijection with $\text{Adm}_c(D)$. Furthermore, if $\sigma \in \text{Adm}(D)$, we can perform crossing resolutions like those preceding \ref{thm:lin_wang} to decompose $D$ into a set of closed curves, each of which is labelled entirely by $0$ or entirely by $1$. Therefore it makes sense to write $\text{rot}_i(D, \sigma)$ for an admissible state $\sigma$.

Finally, we may assume all crossings of $D$ have both strands oriented downward---otherwise, we can apply an isotopy as in Figure \ref{fig:cross_twist}. This isotopy does not change the value of (\ref{eq:lw}), since it does not change the diagram or any rotation numbers. With this assumption, if $p \in D$ is a crossing point, $\omega_p(\sigma) = \omega_p(s)$ for any state $\sigma \in \text{Adm}(D)$ with corresponding state $s \in \text{Adm}_c(D)$.

\begin{figure}[H]
\centering
\includegraphics[height=1.8cm]{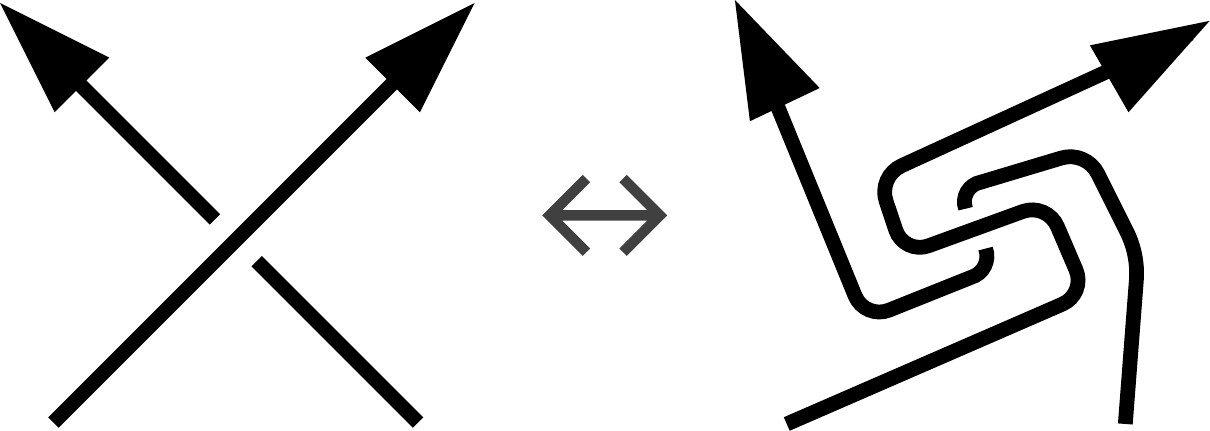}
\caption{An isotopy which reorients a crossing downward}
\label{fig:cross_twist}
\end{figure}

We now compute:
\begin{align*}
J^T_2(L; q) &= (q^{3/4})^{-w(D)} \sum_{\sigma \in \text{Adm}(D)} \prod_{p \in P} \omega_p(\sigma) \\
&= (q^{3/4})^{-w(D)} \sum_{\sigma \in \text{Adm}(D)} \prod_{p \in (P \setminus P^c)} \omega_p(\sigma)\prod_{p \in P^c} \omega_p(\sigma) \\
&= (q^{3/4})^{-w(D)} \sum_{\sigma \in \text{Adm}(D)} (\mu_0^0)^{\text{rot}_0(D,\sigma)}(\mu_1^1)^{\text{rot}_1(D,\sigma)} \omega^c(\sigma) \\
&= (q^{3/4})^{-w(D)} \sum_{s \in \text{Adm}_c(D)} q^{(\text{rot}_1(D,s) - \text{rot}_0(D,s))/2} \cdot \omega^c(s).
\end{align*}
The key observation of the third equality is that $\mu$ counts rotation numbers. Examining Theorem \ref{thm:fund_tangles}, we see that a weight of $\mu^i_i$ is assigned to each left-oriented, $i$-colored cap and a weight of $(\mu^{-1})^i_i = (\mu^i_i)^{-1} $ is assigned to each left-oriented, $i$-colored cup. Thus, if $C$ is a curve of $D$ (after crossing resolution) labeled entirely by $i$, the exponent of the product of the $\mu_i^i$s gives the rotation number of $C$. (See Figure \ref{fig:mu_rotation}.)
\end{proof}

\begin{figure}[H]
\labellist
\small\hair 2pt
\pinlabel $\mu$ at 115 350
\pinlabel $\mu$ at 150 200
\pinlabel $\mu^{-1}$ at 640 40
\endlabellist
\centering
\includegraphics[height=4cm]{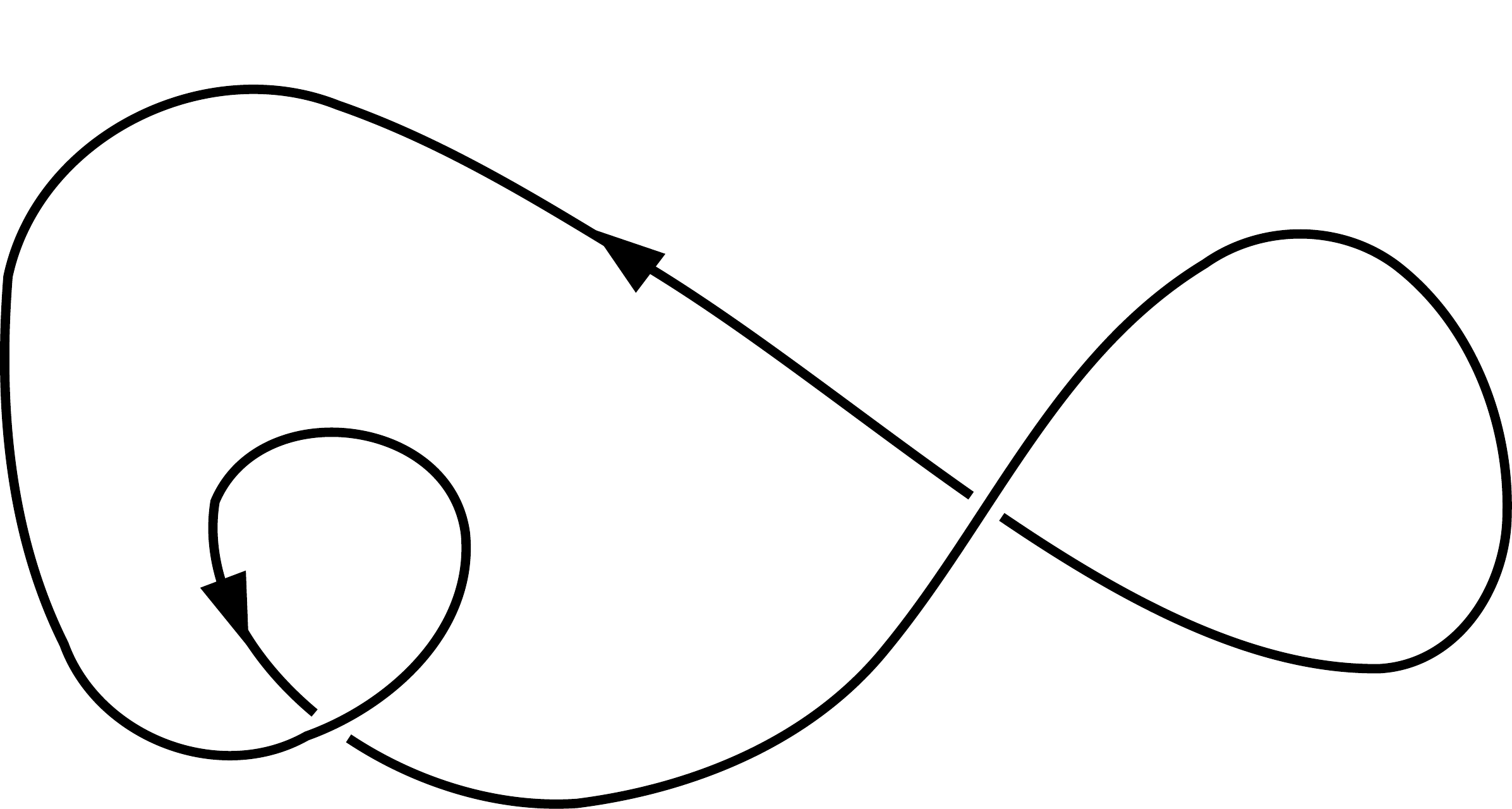}
\caption{The exponent of the product of the $\mu$'s is the rotation number of the curve (in this case 1)}
\label{fig:mu_rotation}
\end{figure}

Having defined $J^T_2$ as in Theorem \ref{thm:me_and_lw}, the higher invariants $J^T_n$, $n > 2$, can be recovered using the Cabling Formula Theorem \ref{thm:cabling}.

As promised, we only needed to normalize the formula in Proposition \ref{thm:lin_wang} to define $J^T_2$ as in Theorem \ref{thm:me_and_lw}. From this perspective $J_n$ and $J^T_n$ become two instances of the same formula, and the definition of the latter is forced by the definition of the former. In other words, from this point of view, there is no other way we could have defined $J^T_n$.

Additionally, (\ref{eq:lw}) provides insight into Proposition \ref{thm:torus_knots}. Let $K \subset T^2 \times I$ be a knot which projects to a simple, closed curve $C \subset T^2$. Then $C$ has no crossings, and only two state assignments as defined in (\ref{eq:lw}). If $C$ is contractible, it has rotation number $\pm1$ and
$$
J^T_2(K;q) = q^{1/2} + q^{-1/2} = [2].
$$
If $C$ is not contractible, it has rotation number $0$ and
$$
J^T_2(K;q) = q^0 + q^0 = 2.
$$
While we cannot fully explain why the toroidal colored Jones polynomial ``quantizes'' contractible curves and not essential ones, this discussion suggests a relationship with the curvature of a link.

\begin{rmk}
The exact $R$-matrix used here is slightly different than the one used in \cite[Sec. 2.3]{lw01}. To recover that matrix from ours, first multiply $R$ by $q^{1/4}$ (and multiply $R^{-1}$ by $q^{-1/4}$), then make the variable substitution $q' = -q^{1/2}$. We also use downward-oriented crossings rather than upward-oriented ones---these two convention changes result in a slightly different formula for $J_2$.
\end{rmk}

\bibliography{volume_conjecture}{}
\bibliographystyle{amsplain}
\end{document}